\documentclass[10pt]{amsart}
\usepackage[totalwidth=440pt, totalheight=640pt]{geometry}
\usepackage{amsmath, amssymb, amsthm, amsfonts, 
eucal, enumerate, natbib, layout}
\usepackage[usenames]{color}
\usepackage{mathtools}
\usepackage{hyperref}
\hypersetup{colorlinks=true, citecolor={red}, linkcolor={blue}, urlcolor={violet}}

\font\cyrfam=wncyr10 scaled 1000
\def\cyr#1{\mbox{\cyrfam\relax#1}}

\newcommand{\details}[1]{}

\newcommand{\longuparrow}{\mathrel{\rotatebox[origin=c]{90}{$\longrightarrow$}}}
\newcommand{\longdownarrow}{\mathrel{\rotatebox[origin=c]{-90}{$\longrightarrow$}}}

\newtheorem{theorem}{Theorem}[section]
\newtheorem*{theorem*}{Theorem}
\newtheorem{corollary}[theorem]{Corollary}
\newtheorem*{corollary*}{Corollary}
\newtheorem{lemma}[theorem]{Lemma}
\newtheorem*{lemma*}{Lemma}

\newtheorem*{claim*}{Claim}

\newtheorem{proposition}[theorem]{Proposition}
\newtheorem*{proposition*}{Proposition}

\newtheorem*{conjecture*}{Conjecture}
\newtheorem{def-proposition}[theorem]{Definition-Proposition}
\theoremstyle{definition}
\newtheorem{definition}[theorem]{Definition}
\newtheorem*{definition*}{Definition}
\newtheorem{remark}[theorem]{Remark}
\newtheorem{notation}[theorem]{Notation}
\newtheorem{example}[theorem]{Example}
\newtheorem*{example*}{Example}
\numberwithin{equation}{section}

\newcommand{\bS}{\mathbf{S}}
\newcommand{\cD}{\mathcal{D}}

\newcommand{\ZZ}{\mathbb{Z}}
\newcommand{\QQ}{\mathbb{Q}}
\newcommand{\RR}{\mathbb{R}}
\newcommand{\CC}{\mathbb{C}}
\newcommand{\GG}{\mathbb{G}}

\newcommand{\LL}{\mathbb{L}}

\newcommand{\End}{\mathrm{End}}
\newcommand{\Hom}{\mathrm{Hom}}

\newcommand{\Aut}{\mathrm{Aut}}

\newcommand{\UR}{\mathrm{UR}}
\newcommand{\W}{\mathrm{W}}

\newcommand{\T}{\mathrm{T}}

\newcommand{\HH}{\mathrm{H}}

\newcommand{\Lie}{\mathrm{Lie}\,}

\newcommand{\dR}{\mathrm{dR}}

\newcommand{\Galmot}{{\mathcal{G}}{\mathrm{al}}_{\mathrm{mot}}}
\newcommand{\MT}{\mathrm{MT}}

\newcommand{\tlog}{\widetilde{\log}}

\newcommand{\tR}{\widetilde{R}}

\newcommand{\tors}{\mathrm{Tors}}
\newcommand{\oK}{\overline K}
\newcommand{\oQQ}{\overline{\QQ}}
\newcommand{\gal}{{\mathrm{Gal}}(\oK/K)}
\newcommand{\sym}{{\mathrm{sym}}}

\newcommand{\cP}{\mathcal{P}}
\newcommand{\cE}{\mathcal{E}}

\begin{document}

\title[Weil pairing]
{Mumford-Tate groups of 1-motives and Weil pairing}

\author{Cristiana Bertolin}
\address{Dipartimento di Matematica, Universit\`a di Padova, Via Trieste 63, Padova, Italy}
\email{cristiana.bertolin@unipd.it}

\author{Patrice Philippon}
\address{\'Equipe de Th\'eorie des Nombres, Institut de Math\'ematiques de Jussieu-Paris Rive Gauche, UMR CNRS 7586, Paris, France}
\email{patrice.philippon@upmc.fr}

\subjclass[2010]{14F35,14F20,11G99}

\keywords{abelian variety, motivic Weil pairing, biextensions, 1-motives, motivic Galois group, unipotent radical, Mumford-Tate group}




\begin{abstract}
	We show how the geometry of a 1-motive $M$ (that is existence of endomorphisms and relations between the points defining it) determines the dimension of its motivic Galois group $\Galmot(M)$.
	 Fixing periods matrices $\Pi_M$ and $\Pi_{M^*}$ associated respectively to a 1-motive $M$ and to its Cartier dual $M^*,$ we describe the action of the Mumford-Tate group of $M$ on these matrices. In the semi-elliptic case, according to the geometry of $M$ we classify polynomial relations between the periods of $M$ and we compute exhaustively the matrices representing the Mumford-Tate group of $M$. This representation brings new light on Grothendieck periods conjecture in the case of 1-motives.
\end{abstract}


\maketitle



\section*{Introduction}
 
 Let $K$ be a subfield of $\CC$ and denote by $\oK$ its algebraic closure in $\CC$.
Consider a 1-motive $M$ defined over $K$, that  is a length one complex
 $M=[u: X\rightarrow G]$, where $X$ is the character group of an $n$-dimensional torus,
 $G$ is an extension of an abelian variety $A$ by an $s$-dimensional torus which is parametrized by $s$ points of the Cartier dual of $A$
 \begin{equation}\label{eq:Q}
 Q_1,\dots,Q_s \in A^*  (\oK) 
 \end{equation}
  (all tori, $G$ and $A$ are defined over $K$), and  $u: \ZZ^n \to G (\oK)$ is a $\mathrm{Gal}(\oK/K)$-equivariant group morphism defined by $n$ points 
 \begin{equation}\label{eq:R}
 	 R_1,\dots,R_n \in G  (\oK) 
 	\end{equation}
 living respectively above $n$ points 
  \begin{equation}\label{eq:P}
  	 P_1,\dots,P_n \in A  (\oK) .
 \end{equation}

Denote by $B$ the abelian sub-variety of $A^n \times A^{*s}$ which is generated by the point $(P_1, \dots, P_n,Q_1,\dots \\, Q_s)$ modulo isogenies. Let $Z(1)$ be the smallest sub-torus of $\GG_m^{ns}$ (after extension of the base field $K$ to its algebraic closure if necessary) which contains the image of the Lie bracket $[\cdot,\cdot]: B \otimes B \to \GG_m^{ns}$ constructed in \eqref{[--]} using the motivic Weil pairing of $A$ (see \cite[\S 1.3]{B03})
$$\mathcal{W}_A^{\;\cP} : A \otimes A^* \longrightarrow \GG_m, $$ and the point $ \pi (pr_*\tR)$ constructed in Subsection \ref{subsection:BZ} (see in particular \eqref{R'} and \eqref{piR'}).

Let $ \Galmot (M)$ be the motivic Galois group of $M$. Because of the weight filtration, $ \Galmot (M)$ fits into 
the following exact sequence 
\begin{equation}\label{eq:shortexactsequenceUR}
	 0 \longrightarrow \UR(M) \longrightarrow \Galmot (M) \longrightarrow \Galmot (A) \longrightarrow 0
\end{equation} 
where $\UR(M)$ is its unipotent radical and $\Galmot (A)$ is the motivic Galois group of $A$, which is its largest reductive quotient (see for example \cite[\S 3.1]{BPSS}).
By \cite[Théorème 0.1]{B03} 
the Lie algebra of $\UR(M)$ is the semi-abelian variety
\begin{equation}\label{eq:ses}
	0 \longrightarrow Z(1) \longrightarrow  \Lie \UR (M)  \longrightarrow B \longrightarrow 0
\end{equation} 
defined by the adjoint action of the Lie algebra $(B,Z(1),[\cdot,\cdot])$ on $B+Z(1)$. In \cite[\S 4]{J} Jossen gives a more elementary construction of the semi-abelian variety $\Lie \UR(M).$

Our first aim is to understand how the geometry of $M$ (existence of endomorphisms and relations between the points \eqref{eq:Q}, \eqref{eq:P} and \eqref{eq:R}) governs the dimension of $\UR(M)$, assuming $K=\oK$.
Because of the short exact sequence \eqref{eq:ses}, the study of $\UR(M)$ reduces to the study of the group varieties $B$ and $Z(1)$. 

For example, if the dimension of $B\subseteq A^n \times A^{*s} $ is maximal, that is $(n+s) g,$ 
where $g$ is the dimension of the abelian variety $A,$
by Corollary \ref{cor:maxdim} the dimension of the unipotent radical $\UR(M)$ is maximal too and equal to
\[ \dim_\QQ \UR(M) = 2 (n+s) g + ns . \]
The dimension of $B$ decreases when there are relations between $ P_1,\dots,P_n, Q_1,\dots,Q_s$ defined in \eqref{eq:P} and \eqref{eq:Q}. These relations are given by endomorphisms of $A$, endomorphisms of $A^*$ and
group morphisms between $A$ and $A^*$. If the abelian variety $A$ is simple, as in Example \ref{ex:dimA}, we compute the dimension of $B$ using the abelian logarithms of the points $P_i$ and $Q_j$. One can deal with the general case using \cite[Lemma 4.3]{BPSS}.

The maximal dimension of $Z(1) \subseteq \GG_m^{ns}$ is clearly $ns.$ The torus $Z(1)$ splits in two parts: 
\[Z(1) = Z'(1) \times (Z/Z') (1),\]
where $Z'(1)$ is the smallest sub-torus of $\GG_m^{ns}$ containing the image of the Lie bracket $[\cdot,\cdot]: B \otimes B \to \GG_m^{ns}$, and  $(Z/Z')(1)$ is
the smallest sub-torus of $\GG_m^{ns} /Z'(1) $ containing the point $ \pi (pr_*\tR)$ constructed in \eqref{piR'}.
\begin{itemize}
	\item The dimension of $Z'(1)$ is governed by the abelian variety $B$ since, as we will show in Lemma \ref{lemma:pr*i*d*B-trivial}, $Z'(1)$ contains 
	the values of the factors of automorphy of the torsor $i^*d^*\mathcal{B} $ over $B$ which is defined in Subsection \ref{subsection:BZ}.
	In Theorem \ref{teo:dimZ'(1)} we compute explicitly the dimension of $Z'(1)$ in terms of the isogenies $\beta_{i,k}: B \to B^*$ (for $i=1, \dots, n$ and $k=1, \dots, s$) constructed in \eqref{eq:beta} using the inclusion of $B$ into $A^n \times A^{*s}$.

	\item The dimension of the torus $(Z/Z')(1)$ involves the point $ \pi (pr_*\tR)$ constructed in \eqref{piR'}. More precisely, in Theorem \ref{teo:dimZ/Z'} we prove that the dimension of $(Z/Z')(1)$ is equal to the dimension of the $\QQ\,$-sub-vector space of $\CC / 2 i \pi \QQ$ generated by the logarithms $\sum_{i,k}\alpha_{i,k}\log(t_{i,k})$, with $\sum_{i,k}\alpha_{i,k} x_i\otimes y_k^\vee$ running over ${Z'}^\perp$ as in \eqref{pointpsi} and $(t_{i,k})$ any point projecting onto $ \pi ( pr_*\tR )$ via $\GG_m^{ns} \to \GG_m^{ns} /Z'(1)$.
\end{itemize}

 For a 1-motive $M$ defined over $K$ and for its Cartier dual $M^*$,  we compute explicitly their periods matrices $\Pi_M$ and $\Pi_{M^*}$ (see Subsection \ref{periodsonemotives}). The entries of these matrices consist of abelian integrals of first, second and third kind computed along closed and open paths. In Lemma \ref{computperiodes} we state relations between these abelian integrals. This allows us to generalize to 1-motives the classical relation between the periods matrices of $A$ and of its Cartier dual $A^*$, namely
 $\Pi_{M^*} = (2\mathrm i\pi\Pi_M^{-1})^t$ (see Corollary \ref{cor:lien}).

 Denote by $\MT(M)$ the Mumford-Tate group of a 1-motive $M$.
According to \cite[\S 6]{J}, for $K=\CC$  the unipotent radical of the motivic Galois group of $M$ coincides with the unipotent radical of $\MT(M)$. More generally, 
  by \cite[Theorem 1.2.1]{A19} and assuming $K=\oK \subseteq \CC$ the image $ \T_\HH (\Galmot (M))$ of $\Galmot (M)$ via the fibre functor $ \T_\HH$ ``Hodge realization" coincides with $\MT(M)$.
The periods matrix $\Pi_M$ of $M$ is a complex point of the $\MT(M) \otimes K$-torsor of periods $\underline{\mathrm{Isom}}^\otimes(  \T_\dR, \T_\HH ).$ In \eqref{actionGalmot} we describe the action of $\MT(M)$ on this point $\Pi_M$.
 In Section \ref{lastSection} we study the representation as elements of $\mathrm{Gl}_{n+2g+s}(\mathbb Q)$ of the unipotent radical of the Mumford-Tate group of $M$ (see Proposition \ref{propGRUMTM}). See \cite[Theorem 4.9.2]{Murty} for another explicit characterization of this unipotent radical. 

 For 1-motives defined over $\oQQ$ and of the type $M=[u:\ZZ \to G],$ with $G$ an extension of an elliptic curve $\cE$ by $\GG_m$, in Section \ref{subsubsection4} we list the polynomial relations between the periods of $M$, which are dictated by its geometry. Using Grothendieck periods conjecture,
 we get a parametrization of the matrices of $\mathrm{Gl}_{4}(\mathbb Q)$ which represent the Mumford-Tate group of $M$. This last section extends Bertrand's work
 \cite[\S 2]{Br08}.

The approach of this paper enables to compute the dimension of the unipotent radical of the motivic Galois group of a 1-motive algorithmically: see \S \ref{subsection:B} and in particular Example \ref{ex:dimA}, Theorem \ref{teo:dimZ'(1)} and Theorem \ref{teo:dimZ/Z'}. Recall that
the dimension of $\Galmot(M)$ is a lower bound for the transcendence degree of the field $K(\Pi_M)$ generated by periods of $M$ as predicted by the Grothendieck periods conjecture (generalized if $K$ is not algebraic) (\cite[Appendix 2]{B19}, \cite{B02}, or \cite{BPSS})).
Hence the explicit description of the action of the Mumford-Tate group of a 1-motive on its periods matrix $\Pi_M$, that we have computed in Section \ref{lastSection}, is useful if one wants to attack the Grothendieck periods conjecture (generalized or not) by a classical transcendence approach.

\section*{Acknowledgements}
We are very grateful to Pierre Deligne for his clarifications on the Weil pairing of an abelian variety.
We thank equally Daniel Bertrand for his help at various stages of this paper. Several of our results extend his own works.
We are very thankful to the referee for his careful reading and his comments which allowed to greatly improve our text.

The first author is supported by the project funded by the European Union – Next Generation EU under the National Recovery and Resilience Plan (NRRP), Mission 4 Component 2 Investment 1.1 - Call PRIN 2022 No. 104 of February 2, 2022 of Italian Ministry of University and Research; Project 20222B24AY (subject area: PE - Physical Sciences and Engineering) ``The arithmetic of motives and L-functions".

\section{Notations}\label{notation}

In this paper $K$ is a subfield of the field of complex numbers $\CC$ and $\oK$ is its algebraic closure in $\CC$. The equality $K=\oK$ means that the field $K$ is algebraically closed in $\CC$.

\subsection{Abelian varieties}\label{abelianvariety}

Let $A$ be a $g$-dimensional abelian variety defined over $K \subseteq \CC$. We identify $A$ and its Cartier dual $A^*$ with the complex torii $\CC^g/\Lambda$ and $\Hom_{\overline \CC}(\CC^g,\CC)/\Lambda^*$ respectively, where $\Lambda$ and $\Lambda^*$ are dual lattices for the duality product
\begin{equation}\label{dualityproduct}
	\langle z,z^*\rangle : = \mathrm{Im}\big(z^*(z)\big)
\end{equation}
with $z^*(z) = \sum_{i=1}^g \overline{z}_iz_i^*$ in coordinates for dual bases and $\Hom_{\overline \CC}(\CC^g,\CC)$ is the $\CC$-vector space of $\CC$-antilinear homomorphisms from $\CC^g$ to $\CC$. Recall that 
\[\Lambda^*=\{\lambda^*\in \Hom_{\overline \CC}(\CC^g,\CC);\mathrm{Im}(\lambda^*(\Lambda))\subset\mathbb Z\}.\]
 If we identify the vector space  $ \Hom_{\overline \CC}(\Hom_{\overline \CC}(\CC^g,\CC),\CC)$ modulo $\Lambda^{**}$ with $\CC^g$ modulo $\Lambda,$ the duality product on $A^* \times A^{**}$ is $z(z^*) = \overline{z^*(z)}$ and so $	\langle z,z^*\rangle = - 	\langle z^* , z\rangle$ (see \cite[\S 2.4, page 34]{BL}).

Up to isogenies we may assume that $A$ has a principal polarization $H: \CC^g \times \CC^g \to \CC$. According to \cite[Lemma 2.4.5]{BL} there exists a map
\begin{align}\label{isoPhi_H}
	\phi_H : \CC^g & \longrightarrow  \Hom_{\overline \CC}(\CC^g,\CC),\\
	\nonumber  z & \longmapsto H(z, \cdot)
\end{align}
which is an isomorphism since the polarization is principal. In particular, for $p,q \in \CC^g$ one has
\begin{equation} \label{eq:ImH}
	\langle p, \phi_H(q) \rangle = \mathrm{Im} \big(\phi_H(q)(p) \big)= \mathrm{Im} \big( H(q,p)\big) .
\end{equation}
Occasionally, we identify via the isomorphism $\Phi_H$ in \eqref{isoPhi_H}, $\Hom_{\overline \CC}(\CC^g,\CC)/\Lambda^*$ with   $\mathbb C^g/ \phi_H^{-1}(\Lambda^*).$

Let $\cP$ be the Poincaré biextension of $(A,A^*)$ by $\GG_m$. According to \cite[\S 2.5, page 38 and Appendix B, page 571]{BL}, the sections of $\cP|_{A \times \{Q\}}$ can be viewed as the sections of $\exp_A^*\cP|_{A \times \{Q\}}$ which are invariant under the action of $\Lambda$ through the factor of automorphy $e^{\pi{z^*(\lambda)}}$, where $\exp_{A^*}(z^*)=Q$ and $\lambda\in\Lambda$.  Fix a divisor $D$ not containing the origin and associated to the line bundle with first Chern class $H$ and trivial semi-character.
Let  $\theta(z)$ be a theta function with divisor $D$ such that $\theta (0)=1$ and $d\log\theta(0)=0$. For $q\in\CC^g$ such that $\exp_{A^*}(\phi_H(q))=Q$, the meromorphic function $\Theta_q(z):=\frac{\theta(z+q)}{\theta(z)\theta(q)}$ is a theta function (in $z$) for the canonical factor $e^{\pi H(q,\lambda)}=e^{\pi\phi_H(q)(\lambda)}$, see \cite[\S 2.3]{BL}. It thus induces a meromorphic (that is rational) section of $\cP|_{A\times\{Q\}}$, however depending on the choice of $q$. But the ratio $\frac{\Theta_q(z_1+z_2)}{\Theta_q(z_1)\Theta_q(z_2)}$ is a rational function on $A\times A$ not depending on the choice of the logarithm $q$ of $Q$. Hence the latter is a system of factors defining a group law on $\cP|_{A\times\{Q\}}$. Let $G$ be the extension of $A$ by $\GG_m$ parametrized by the point $Q \in A^*$. In order to write the exponential map of $G$ we introduce the function
\begin{equation}\label{Fq(z)}
	F_q(z):=\Theta_q(z)e^{- \sum_{i=1}^g h_i(q) z_i},
\end{equation}
which is periodic in $q$ with respect to $\Lambda$, see \cite[pages 40-41]{Br83} and \cite[2.3]{BPSS}. This function replaces the function introduced by Serre in the elliptic case, see \cite[(1.6)]{B19}. 
It defines a rational section of the line bundle underlying $G$ and the same system of factors on $A\times A$ as $\Theta_q$. The exponential factor $e^{-\sum_{i=1}^g h_i(q) z_i}$ insures that the differential of the exponential map $\exp_G$ at the origin, is the identity. In summary the exponential maps of $G$ and $\cP|_{A \times \{Q\}}$ are
\begin{align}
	\label{expG}	\exp_G(z,t) & =\big(\exp_A(z), e^{t}F_{q}(z)\big)\\
	\nonumber	\exp_{\cP|_{A \times \{Q\}}}(z,t) &=  \big(\exp_A(z), e^{t} \Theta_q(z)\big)
\end{align}
respectively and so we have the isomorphism
\begin{equation}\label{isoGP}
	\begin{matrix}\hfill G &\to &\cP|_{A \times \{Q\}}\hfill\\[2mm]
	\big(\exp_A(z), e^{t}F_{q}(z)\big)\mod\Lambda &\mapsto & \big(\exp_A(z), e^{t}F_{q}(z)e^{\sum_{i=1}^g h_i(q) z_i}\big)\mod\Lambda.\end{matrix}
\end{equation}
In particular, the fibre $\cP_{P,Q}$ of the Poincaré biextension $\cP$ above the point $(P,Q)$ corresponds to the fibre $G_P$ of the extension $G$ above the point $P$.

\subsection{1-motives}\label{1-motives}

A 1-motive $M=[u:X \rightarrow G]$ over $K \subseteq \CC$ consists of
\begin{itemize}
\item  a character group $X$ of a $n$-dimensional torus defined over $K$, that is a finitely generated torsion free $\gal$-module,
\item	 an extension $G$ of an abelian variety by a torus, which is defined over $K$, and
\item	 a morphism $u:X \to G$ of group varieties over $K$. 
	 \end{itemize}
 Since any morphism of $K$-varieties can be seen as a $\gal$-equivariant morphism  
 of the corresponding $\oK$-varieties, having the group morphism $u:X \to G$ is equivalent to having a $\gal$-equivariant group morphism 
 \[u: \ZZ^n \longrightarrow G(\oK).\]
As in \cite{BPSS,B03,B19,BG2,B11,BB,BG}, we denote $Y(1)$ the torus whose character group is $Y^\vee$ and co-character group is $Y$.

There is a more symmetrical definition of 1-motives: having the 1-motive $M=[u:X \rightarrow G]$ is equivalent to having the 7-tuple $(X, Y^\vee,A ,A^*, v ,v^*,\psi)$ where:
\begin{itemize}
	\item $X$ is the character group of a $n$-dimensional torus $X^\vee (1)$;
	\item $Y^\vee$ is the character group of a $s$-dimensional torus $Y(1)$;
	\item $A^*$ is the Cartier dual of the abelian variety $A$;
	\item $v: X \rightarrow A$ and $v^*:Y^\vee \rightarrow A^*$ are two group morphisms of $K$-group varieties, that is two  $\gal$-equivariant group morphisms $ v: \ZZ^n \rightarrow A(\oK)$ and $v^*: \ZZ^s \rightarrow A^*(\oK)$;
	\item $\psi$ is a trivialization (= biadditive section) of the pull-back $(v \times v^*)^*\mathcal{P}$ via $v \times v^*$ of the Poincar\'e biextension $\mathcal{P}$ of $(A,A^*)$ by $\GG_m$.
\end{itemize}

If we extend the scalars to $\oK$ the tori become split: ${X^\vee} (1) (\oK) \cong \GG_m^n(\oK) $ and ${Y}(1)(\oK)\cong \GG_m^s(\oK).$
Denote by $x_1, \dots, x_{n}$ (\textit{resp}. $y^\vee_1, \dots , y^\vee_{s}$) the basis of the character group $X$ (\textit{resp.} $Y^\vee$) giving the identification of $X$ with $\ZZ^n$ (\textit{resp}. of $Y^\vee$ with  $\ZZ^s$) over $\oK$ and set
\[v(x_i)=P_i \in A(\oK), \qquad v^*(y^\vee_k) =Q_k \in A^*(\oK)\]
for $ i=1, \dots,n $ and $ k=1, \dots,s. $ Having the group morphism $v$ is equivalent to having the $n$ points $P_1, \ldots, P_n\in A(\oK)$, whereas having the group morphism $v^*$ is equivalent to having the $s$ points $Q_1, \ldots, Q_s\in A^*(\oK)$ which parametrize the extension $G$ of $A$ by $Y(1)$.

 For $k=1, \dots,s,$ denote by $G_{Q_k}$
the push-down of the extension $G$ via the character $y^\vee_k : Y(1) \to \GG_m$: it is an extension of $A$ by $\GG_m$ parametrized by the point $Q_k \in A^*$. By \cite[Chap. VII, Proposition 1]{Se84}
the extension $G$ is isomorphic to $G_{Q_1}\times_A\dots\times_A G_{Q_s}$.
 Therefore for $i=1, \dots , n$ having the point $ u(x_i)=R_i $ in the fibre $G_{P_i}(\oK)$ of $G$ above the point $v(x_i)=P_i \in A(\oK)$ is equivalent to having the $s$ points $R_{i,1}, \dots , R_{i,s}$ in the fibres $(G_{Q_1})_{P_i}(\oK) , \dots , (G_{Q_s})_{P_i}(\oK)$ respectively. Via the isomorphism \eqref{isoGP}, we have that
\[R_{i,k} = \psi(x_i,y^\vee_k) \in \cP_{P_i,Q_k} \simeq (G_{Q_k})_{P_i} \]
for $k=1, \dots,s$.  We can conclude that
having the trivialization $\psi$ is equivalent to having the $n$ points $R_1, \ldots, R_n \in G(\oK)$, whose images via the projection $G \twoheadrightarrow A$ are the $n$ points $P_1, \ldots, P_n$ respectively, i.e. it is equivalent to having the group morphism $u:X \rightarrow G$. 

 \textit{The Cartier dual of $M$} is the 1-motive $M^*= [u^*:Y^\vee \rightarrow G^*]=( Y^\vee,X,A^*,A, v^* ,v,\psi),$ where
\begin{itemize}
	\item $G^*$ is the extension of $A^*$ by  $X^\vee(1)$ parametrized by the $n$ points $P_1,\dots,P_n $ of $ A(\oK)$ given by the $\gal$-equivariant group morphism $v: \ZZ^n \to A(\oK),$ and 
	\item  $u^*:\ZZ^s \rightarrow G^*(\oK)$ is the $\gal$-equivariant group morphism given by the $s$ points $ R_1^*,\dots,R_s^* \in G^*(\oK)$
	defined by the trivialization $\psi$ of the pull-back $(v \times v^*)^*\mathcal{P}$ via $(v \times v^*)$ of the Poincar\'e biextension $\mathcal{P}$ of $(A,A^*)$ by $\GG_m$ (see \cite[(10.2.12.2)]{D74}).
\end{itemize}

\section{The motivic Weil pairing of an abelian variety}\label{geometricaldescriptionW_A}

\begin{definition} Let $A$ be an abelian variety defined over $K   \subseteq \CC$.
	The isomorphism class of the Poincaré biextension $\cP$ of $(A,A^*)$ by $\GG_m$ is the motivic Weil pairing $ \mathcal{W}_A^{\;\cP} :A \otimes A^*  \rightarrow \GG_m$  of $A$ in a suitable tannakian category of motives.
\end{definition}

\begin{remark}
	Beware that until now nobody has proved that $A\otimes A^*$ is representable by a group variety and in particular one cannot consider the value of this arrow on points of  $A \otimes A^*$.
\end{remark}

In this paper we use the 
homological interpretation and the Hodge realization of this motivic Weil pairing of $A$, that we investigate further in the following Subsections.
See \cite[Construction (10.2.3),(10.2.5),(10.2.7)]{D74} or \cite{BM09,B12,B13,BT14,BT18} for the definition of the pairings 
$ {\T}_*(A)\otimes {\T}_*(A^*)  \rightarrow  {\T}_*(\GG_m)$
in each cohomology theory $*= \HH, \dR,\ell$ and recall that the $\ell$-adic realization $ {\T}_\ell(A)\otimes {\T}_\ell(A^*)  \rightarrow  {\T}_\ell(\GG_m)$ is just the classical $\ell$-adic Weil pairing of $A$.

\medskip

\begin{lemma}\label{propieteW_A}
The motivic Weil pairing $	\mathcal{W}_A^{\;\cP}:A  \otimes A^*  \rightarrow \GG_m$ has the following properties:
	\begin{enumerate}
		\item bilinearity, 
		\item $\mathcal{W}_{A^*}^{\;\cP^*}= (\mathcal{W}_{A}^{\;\cP})^{-1} \circ \sym$, where $\sym: A^*  \otimes A \to A  \otimes A^*$ is the arrow which exchange the factors,
		\item non-degenerate,
		\item any group morphism $f: B \to A$ of abelian schemes 
		and its transpose $f^t: A^* \to B^*$ are adjoint with respect to the motivic Weil pairings, that is $ \mathcal{W}_A^{\;\cP_A} \circ (f \otimes id_{A^*}) = \mathcal{W}_B^{\;\cP_B} \circ ( id_B \otimes f^t) . $ 
	\end{enumerate}
\end{lemma}

\begin{proof} 
	$(1)$ The bilinearity of the arrow $\mathcal{W}_A^{\;\cP}$ is a consequence of the two group laws on the Poincaré biextenion $\cP$. 
	
	$(2)$ Consequence of \cite[Remarque (10.2.4)]{D74}.
	
	$(3)$ Since $\cP$ is trivial only over $\{e\} \times A^*$ and over $A \times \{e^*\},$ where $e$ (\textit{resp.} $e^*$) is the neutral element of $A$ (\textit{resp.} $A^*$), we have the statement (see also \cite[\S 2.5 page 38]{BL}).
	
	$(4)$ It follows by \cite[Lemma 3.2.1]{B09}.
\end{proof}

\subsection{Homological interpretation}
Let $A$ be an abelian scheme over an arbitrary scheme $S$.
 By \cite[Expos\'e VII Corollary 3.6.5]{SGA7}, the isomorphism class of the Poincaré biextension $\cP$ of $(A,A^*)$ by $\GG_m$ defines
 an arrow $\mathcal{W}_A:A {\buildrel {\scriptscriptstyle \LL}\over \otimes} A^*   \rightarrow  \GG_m[1]$ in the derived category $\cD(\bS)$ of complexes of abelian sheaves on the site associated to the scheme $S$ for any Grothendieck topology
 (here $ {\buildrel {\scriptscriptstyle \LL}\over \otimes}$ is the derived functor of the functor tensor product $ \otimes$). Since via Deligne's equivalence of categories \cite[\S 1.4]{SGA4} the complex $\GG_m[1]$ corresponds to the strictly commutative Picard stack of $\GG_m$-torsors, this arrow can be made explicit. 
	For $(P,Q) \in A \times A^*$ we set
	\begin{equation}\label{W_A}
	 	\mathcal{W}_A (P,Q) : = \cP_{P,Q} \in \GG_m[1]
	\end{equation}
and then we extend it by bilinearity to $A \otimes A^*.$ The properties of the Poincaré biextension imply that this arrow factors through $ A {\buildrel {\scriptscriptstyle \LL}\over \otimes} A^* $, giving the expected arrow  $\mathcal{W}_A :A {\buildrel {\scriptscriptstyle \LL}\over \otimes} A^*   \rightarrow  \GG_m[1]$ of $\cD(\bS)$.

	\subsection{Hodge realization} 
			 Assume the abelian variety $A$ to be $g$-dimensional  and defined over $K \subseteq \CC$.
		  As we have done in Section \ref{notation}, we identify $A$ and its Cartier dual $A^*$ with the complex torii $\CC^g/\Lambda$ and $\Hom_{\overline \CC}(\CC^g,\CC)/\Lambda^*$ respectively. The Hodge realization of $A$ is 
		  $\T_\HH (A) = \HH_1(A(\CC), \QQ) \cong \Lambda \otimes_\ZZ \QQ$ and the restriction of the duality product \eqref{dualityproduct} to $\Lambda \times \Lambda^*$ gives rise to a map
		\begin{align}
\nonumber	  \; \Lambda \otimes_\ZZ \Lambda^* &\longrightarrow   \ZZ \\
		\nonumber	 \;(\lambda ,\lambda^*) &\longmapsto  \langle\lambda,\lambda^*\rangle = \mathrm{Im}\big(\lambda^*(\lambda)\big).
	\end{align}

	The \textit{Hodge realization of the motivic Weil Pairing} of $A$ is
	\begin{align}
		\label{eq:hodge} \mathcal{W}_A^\HH:  \; \Lambda \otimes_\ZZ \Lambda^* &\longrightarrow   \ZZ \\
		\nonumber	 \;(\lambda ,\lambda^*) &\longmapsto \langle\lambda,\lambda^*\rangle =  \mathrm{Im}\big(\lambda^*(\lambda)\big).
	\end{align}
	Extending this map by bilinearity after tensoring with $\RR$, we get
	the \textit{Lie realization of the motivic Weil pairing} of $A$
		\begin{equation}\label{def:LieW_ALie}
		\begin{matrix}
			w_A^{\Lie}: & \Lie A_\CC \otimes \Lie A^*_\CC & \longrightarrow
			& \Lie {\GG_m}_\CC\hfill\\
			&(z,z^*) &\longmapsto &2{\rm i}\pi\langle z,z^*\rangle.
		\end{matrix}
	\end{equation}
In this paper we will need the composition of the Lie realization of $\mathcal{W}_A^{\;\cP}$ with the exponential map of $\GG_m,$ that we will denote
	\begin{equation}\label{def:LieW_A}
	\begin{matrix}
		\mathcal{W}_A^{\Lie}: & \Lie A_\CC \otimes \Lie A^*_\CC & \longrightarrow
		& \GG_m (\CC) \hfill\\
		&(z,z^*) &\longmapsto & e^{	w_A^{\Lie}(z,z^*)}=e^{2{\rm i}\pi\langle z,z^*\rangle}.
	\end{matrix}
\end{equation}

If we write $d\big(\log(\theta(z))\big) = \sum_{i=1}^g h_i(z) d z_i$, a base of differential forms of the second kind on $A$ is given by forms $\eta_1, \dots , \eta_g$ such that $\exp_A^*(\eta_i)= dh_i(z) $  for $i=1, \dots,g$ (see \cite[\S 2.2]{BPSS}). Identify $\CC^g$ with $\Lambda \otimes_\ZZ \RR$ through a decomposition of $\Lambda$ for the principal polarization $H$, see \cite[\S 3.1, page 46]{BL}. We extend the quasi-periods map $ \Lambda \to \CC, \lambda\mapsto \int_{\gamma_\lambda}\eta_i = h_i(z+\lambda)-h_i(z)$ to a function $\widetilde h_i :\CC^g\simeq \Lambda\otimes_{\mathbb Z}\mathbb R \to \CC$ by $\RR$-linearity (here $\gamma_\lambda$ is the image by the abelian exponential $\exp_A$ of a path from $0$ to $\lambda$ in $\mathrm{Lie}A_{\CC}$). Consider the function
\begin{align}
\nonumber \widetilde{F}:	(\CC^g\setminus \mathrm{div}(\theta))^2 & \longrightarrow  \CC\\
	\nonumber       (p,q) & \longmapsto \widetilde{F}_{q}(p) := \frac{\theta(p+q)}{\theta(p)\theta(q)}e^{-\sum_{i=1}^g p_i\widetilde h_i(q)}
\end{align}
which is a modification of the function \eqref{Fq(z)}. Observe that the ratio $F_{q}(z) / \widetilde{F}_{q}(z)$, which is well-defined on $\CC^{2g}$, is just the exponential function $ e^{\sum_{i=1}^g ( \widetilde h_i(q) -h_i(q)) z_i}$, which is a trivial theta function in $z$.
  If $p$ and $q$ belong to $\CC^g $, we have
\begin{align}\label{eq:ratio}
	\frac{\widetilde{F}_{q}(p)}{\widetilde{F}_p(q)}
	&= e^{\sum_{i=1}^g( q_i\widetilde h_i(p)- p_i\widetilde h_i(q))}\\
\nonumber	&= e^{ \pi ( H(q,p) -  H(p,q)) }= e^{ 2{\rm i} \pi \mathrm{Im}(H(q,p)) } = e^{ 2{\rm i} \pi\langle p, \phi_H(q) \rangle  } = \mathcal{W}_A^{\Lie}(p,\phi_H(q))
\end{align}
by the following lemma:

\begin{lemma}\label{lemcroixperiodquasi} 
	For $p,q \in \CC^g$, we have the formula $ \sum_{i=1}^g q_i\widetilde h_i(p)  = \pi H(q,p) = \pi \phi_H (q)(p).$ In particular,
	for $\lambda, \mu \in \Lambda$, $ \pi \phi_H (\mu)(\lambda) - \pi \phi_H (\lambda)(\mu)  = 2{\rm i} \pi \mathrm{Im} \big(H(\mu,\lambda)\big) \in 2{\rm i} \pi\ZZ$.
\end{lemma}
\begin{proof} 
	Since $\CC^g \cong \Lambda \otimes_\ZZ \RR $, it is enough to establish the first equality for $(p,q)=(\lambda, \mu) \in \Lambda^2.$
	By definition $\widetilde h_i(\lambda)= h_i(z+\lambda) - h_i(z)$ for any point $z \in \CC^g$ such that neither $z$ nor $z+\lambda$ is a zero of $\theta(z).$
	Furthermore, $h_i(z) = \frac{d}{d z_i} \log \theta(z).$
	By the functional equation satisfied by the function $\theta(z)$, we have
	$$\widetilde{h}_i(\lambda) = \frac{d}{d z_i}\left(\log\left(\frac{\theta(z+\lambda)}{\theta(z)}\right)\right) = \frac{d}{d z_i}\left(\pi H(z,\lambda)+\frac{\pi}{2} H(\lambda,\lambda)\right).$$
	Then, we get the expected formula
	$$  \pi \phi_H (\mu)(\lambda) = \pi H(\mu,\lambda) = \pi\sum_{i=1}^g \mu_i \frac{d}{d z_i}\big(H(z,\lambda)\big)  = \sum_{i=1}^g \mu_i \widetilde{h}_i(\lambda),$$
	since $H(z,\cdot)$ is linear in $z$.
	As a consequence we have the second equality
	$$ \pi \phi_H (\mu)(\lambda) - \pi \phi_H (\lambda)(\mu) = \pi\big(H(\mu,\lambda)-H(\lambda,\mu)\big) = 2\mathrm i\pi\mathrm{Im}\big(H(\mu,\lambda)\big) \in 2{\rm i} \pi\ZZ$$
	since $H(\lambda,\mu) = \overline{H(\mu,\lambda)}$ and $\mathrm{Im}\big(H(\Lambda,\Lambda) \big) \subseteq \ZZ $.
\end{proof}

\begin{remark}
As observed by Deligne in \cite[Lemma (10.2.3.4)]{D74} there is a link between the Hodge realization of the motivic Weil pairing of $A$ and trivializations of the restriction to $\Lambda \times \Lambda^* $  of the pull-back $\exp_{A \times A^*}^* \cP$ on $\Lie A_\CC \times \Lie A^*_\CC$. In fact,
by the above Lemma, for any pair $(z,z^*) \in \Lambda  \times \Lambda^*$, the logarithm
$ \log \Big( \frac{\widetilde{F}_{\phi_H^{-1}(z^*)}(z)}{\widetilde{F}_{z}(\phi_H^{-1}(z^*))} \Big) =  2{\rm i} \pi\langle z, z^* \rangle $
takes values in $ 2 \mathrm{i} \pi  \ZZ$: it is the Hodge realization \eqref{eq:hodge} of the motivic Weil pairing of $A$.
Thus the restriction to $\Lambda  \times \Lambda^*$ of
the ratio $\widetilde{F}_{\phi_H^{-1}(z^*)}(z)/\widetilde{F}_{z}(\phi_H^{-1}(z^*))$ is a
trivialization of the restriction to $\Lambda \times \Lambda^* $  of $\exp_{A \times A^*}^* \cP$.
\end{remark}

\section{Construction of the Lie algebra of the unipotent radical of $ \Galmot (M)$}\label{constructionUR}

Let $M=[u:X \rightarrow G]$ be a $1$-motive over $K   \subseteq \CC$.
 For $i=1, \dots,n$ and $k=1, \dots,s,$ denote by $s_{i,k}:\GG_m \to (X^\vee \otimes Y)(1)$ the sections which describe the torus $(X^\vee \otimes Y)(1)$ as the direct sum of its factors (after extension of the base field $K$ to its algebraic closure if necessary).
 In \cite[Example 2.8]{B03},  the first author has constructed a biextension
 $\mathcal{B}$
 of $(A^n\times A^{*s},A^n\times A^{*s})$ by $(X^\vee \otimes Y)(1)$, whose fibre above the point $((Z_1,Z_1^*),(Z_2,Z_2^*)),$ with  $(Z_\ell,Z_\ell^*) =(Z_{\ell 1}, \dots,Z_{\ell n}, Z_{\ell 1}^*, \dots,Z_{\ell s}^* ) \in A^n \times A^{*s}$  for $\ell=1,2$, is
\begin{equation}
\label{eq:fibres}\mathcal{B}_{(Z_1,Z^*_1), (Z_2,Z^*_2)} 
=\bigwedge_{i=1, \dots,n \atop k=1,\dots,s}  (s_{i,k})_*(\cP_{Z_{1i},Z^*_{2k}} \wedge \cP^*_{Z^*_{1k},Z_{2i}}) =	\Big(\cP_{Z_{1i},Z^*_{2k}} \wedge \cP^*_{Z^*_{1k},Z_{2i}}\Big)_{i=1, \dots,n \atop k=1, \dots s}  ,
\end{equation}
where $\wedge$ is the contracted product.
Its pull-back via the diagonal morphism $d:A^n\times A^{*s} \to( A^n\times A^{*s})^2,$ denoted by
\begin{equation} \label{d^*B}
d^* \mathcal{B},
\end{equation}
is a $\Sigma-(X^\vee \otimes Y)(1)$-torsor over $A^n\times A^{*s}$ (see \cite[2.2]{B03} and in particular \cite[Definition 5.4 and Theorem 8.9]{Breen} for the notion of $\Sigma$-torsor). According to \cite[Remarque (10.2.4)]{D74}
 the Poincar\'e biextensions $\cP$ of $(A,A^*)$ and $\cP^*$ of $(A^*,A)$ satisfy $\cP^*= ( [-{\mathrm {Id}}_{A \times A^*} ] \circ \mathrm{sym})^*{\cP}$, with $\mathrm{sym}: A^* \times A\to A\times A^*$ the group morphism which exchanges the factors. Recall that in Section \ref{notation} we have identified  $A$ with $\CC^g/\Lambda$ and $A^*$ with
 $\Hom_{\overline \CC}(\CC^g,\CC)/\Lambda^*$. If we denote by $a(\lambda,\lambda^*,z,z^*)$ the factor of automorphy of the line bundle underlying $\cP$ (see \cite[\S 2.5, page 38]{BL}), the factor of automorphy of the $(X^\vee \otimes Y)(1)$-torsor $d^*\mathcal B$ is then the morphism
 $\Lambda^n \times \Lambda^{*s} \times \Lie A^n_\CC\times\Lie A^{*s}_\CC \rightarrow (X^\vee \otimes Y)(1) $
 with $(i,k)$-th component
\begin{equation}\label{eq:section}
	a(\lambda_i,\lambda^*_k,z_i,z^*_k){\overline{a(\lambda_i,\lambda^*_k,z_i,z^*_k)}}^{-1} =  e^{2{\rm i}\pi{\rm Im}( \overline{{\lambda}_k^*(z_i) } + z^*_k(\lambda_i))}
\end{equation}
for $i = 1, \dots, n$, $k= 1, \dots , s$.

The Lie bracket 
\begin{equation}\label{eq:bracket}
	[\cdot,\cdot]:( A^n\times A^{*s}) \otimes (A^n\times A^{*s}) \longrightarrow  (X^\vee \otimes Y)(1),
\end{equation}
introduced in \cite[(2.8.4), Lemma 3.3 and Corollaire 2.7]{B03}, has the following interpretations:

 $\bullet$ according to \eqref{W_A} its homological interpretation is the Lie bracket
	$[\cdot,\cdot]:( A^n\times A^{*s}) \otimes^\LL (A^n\times A^{*s}) \to  (X^\vee \otimes Y)(1)[1]$ given by 
	\begin{align}
	\label{[--]}	 \big[(Z_1,Z_1^*), (Z_2,Z_2^*)\big]  & :=	\Big(\mathcal{W}_A (Z_{1 i},Z^*_{2k}) \wedge \mathcal{W}_{A^*}  (Z^*_{1k},Z_{2i})\Big)_{{i=1, \dots,n} \atop {k=1, \dots s}} \\
\nonumber		 & \; =	\Big(  \cP_{Z_{1i},Z^*_{2k}} \wedge \cP^*_{Z^*_{1k},Z_{2i}} \Big)_{i=1, \dots,n \atop k=1, \dots s} 
	\end{align}
where $(Z_\ell,Z^*_\ell) \in A^n \times A^{*s}$ for $\ell=1,2$. The point
	$[(Z_1,Z_1^*), (Z_2,Z_2^*)] $ of $(X^\vee \otimes Y)(1)[1]$ is the fibre \eqref{eq:fibres}
	of $\mathcal{B}$ above the point $((Z_1,Z_1^*), (Z_2,Z_2^*))$ of $( A^n \times A^{*s})^2.$

 $\bullet$ according to \eqref{def:LieW_A} its Lie realization is the Lie bracket  $[\cdot,\cdot]: (\Lie A_\CC^n\times \Lie A^{*s}_\CC)^2 \to (X^\vee \otimes Y)(1)$ defined as
 \begin{align}\label{eq:LieRealization[,]}
 	\big[(z_1,z_1^*),(z_2,z_2^*)\big]_{i,k} &= \mathcal{W}_A^\Lie(z_{1i},z^*_{2k})  \cdot \mathcal{W}_{A^*}^\Lie (z^*_{1k},z_{2i}) \\
 \nonumber	&= \mathcal{W}_A^\Lie(z_{1i},z^*_{2k}) \cdot \mathcal{W}_A^\Lie(z_{2i},z^*_{1k})^{-1}
 	 =  e^{2{\rm i}\pi{\rm Im}( z^*_{2k} (z_{1i}) - z^*_{1k} (z_{2i}) )}
 \end{align}
where $(z_\ell,z^*_\ell) \in \Lie A_\CC^n \times \Lie A^{*s}_\CC$ for $\ell=1,2$.

\subsection{The pure motives underlying the unipotent radical of $\Lie \Galmot (M)$}\label{subsection:BZ}

Recall the notations of Subsection \ref{1-motives}.
Consider $P=(P_1,\dots,P_n) \in A^n, Q =(Q_1,\dots,Q_s) \in A^{*s},$ and $B$
\textit{the smallest abelian sub-variety (modulo isogenies) of $A^n\times A^{*s}$}
 which contains a multiple of the point $(P,Q) \in (A^n\times A^{*s})(\oK).$ 
 Now we restrict the base of the $(X^\vee \otimes Y)(1)$-torsor $d^* \mathcal{B}$ to the abelian variety $B$ by taking the pull-back 
 \begin{equation}\label{i^*d^*b}
 	i^*d^* \mathcal{B}
 \end{equation}
 of $d^* \mathcal{B}$
via the inclusion $i: B \hookrightarrow A^n\times A^{*s}$.

\medskip

Denote by $x_1, \dots, x_{n}$ (\textit{resp}. $y^\vee_1, \dots , y^\vee_{s}$) the basis of the character group $X$ (\textit{resp.} $Y^\vee$) giving the identification of $X$ with $\ZZ^n$ (\textit{resp}. $Y^\vee$ with  $\ZZ^s$) and set
\[v(x_i)=P_i \in A(\oK), \qquad v^*(y^\vee_k) =Q_k \in A^*(\oK)\]
for $ i=1, \dots,n $ and $ k=1, \dots,s. $ 
 Having the group morphism
$v: X \rightarrow A$ (\textit{resp}. $v^*: Y^{\vee} \rightarrow A^*$) is equivalent to having the group morphism
$$V: \ZZ \longrightarrow A^n,\;  V(1)= (P_1, \dots, P_n), \quad  ({resp}. \quad  V^*: \ZZ \longrightarrow  A^{*s},\;  V^*(1)= (Q_1, \dots, Q_s)).$$
Since by \cite[Corollary 2.1.3]{B09bis} $\mathrm{Biext}^1 (X,Y^\vee; \GG_m) \cong \mathrm{Biext}^1 (\ZZ,\ZZ; (X^\vee \otimes Y)(1)),$ the trivialization
$\psi: X \times Y^\vee \to (v \times v^*)^* \cP$, which is equivalent to the group morphism $u: X \to G$, gives a trivialization 
\begin{equation}\label{eq:Psi}
 \Psi : \ZZ \times \ZZ \longrightarrow  (V \times V^*)^* i^*d^* \mathcal{B}
\end{equation}
 of the pull-back by $ V \times V^*: \ZZ \times \ZZ \to B $ of the torsor $i^*d^* \mathcal{B}$.
    In view of \eqref{eq:fibres}, the trivialization $\Psi$ defines  
a point 
\begin{equation}\label{R}
 \Psi(1,1) : =  \big(\psi(x_i,y^\vee_k)\big)_{i= 1, \dots, n \atop k= 1, \dots, s }  \in   \big( (V \times V^*)^* i^*d^*\mathcal{B}\big)_{1,1} 
\end{equation}
which in turn furnishes a point
 $\tR$
  in the fibre of $i^*d^*\mathcal{B}$ over the point $(P,Q) \in  B \subseteq A^n\times A^{*s}.$ 

\medskip

Let $Z'$
be \textit{the smallest $\gal$-sub-module of $X^\vee \otimes Y$} such that the sub-torus $Z'(1)$ of $(X^\vee \otimes Y)(1)$ contains the image of the Lie bracket $[\cdot,\cdot]: B \otimes B \to (X^\vee \otimes Y)(1)$. Consider the push-down ${pr}_*i^*d^* \mathcal{B}$ via the projection $pr:(X^\vee \otimes Y)(1) \twoheadrightarrow (X^\vee \otimes Y/Z')(1)$ of the torsor $i^*d^* \mathcal{B}$.

	 \begin{lemma}\label{lemma:pr*i*d*B-trivial}
	The torus $Z'(1)$ is the smallest sub-torus of $(X^\vee \otimes Y)(1)$ which contains the values of the factor of automorphy of the  $(X^\vee \otimes Y)(1)$-torsor $i^*d^*\mathcal{B}$. In particular
the torsor ${pr}_*i^*d^* \mathcal{B}$ is \textit{the trivial $(X^\vee \otimes Y/Z')(1)$-torsor over $B$}, i.e. 
${pr}_*i^*d^* \mathcal{B}= B \times(X^\vee \otimes Y/Z')(1).$
\end{lemma}
	 
\begin{proof} 
 	 By construction the torus $Z'(1)$ is the smallest sub-torus of $(X^\vee \otimes Y)(1)$ containing the image of the Lie bracket $[\cdot,\cdot]$ restricted to $ \Lie B_\CC \otimes \Lie B_\CC$. So by \eqref{eq:LieRealization[,]} we have that 
	 	\[ Z'(1 ) = \left\langle
	 	 \big(
	 	e^{2{\rm i}\pi{\rm Im}( z^*_{2k} (z_{1i}) - z^*_{1k} (z_{2i}) )} \big)_{i=1, \dots,n \atop k=1, \dots, s}
	 	\; \big\vert \;  (z_1,z_1^*),(z_2,z_2^*) \in \Lie B_\CC  \right\rangle \subseteq  (X^\vee \otimes Y)(1).
	 	\]
	 	Taking $(z_1,z_1^*) = (z,z^*)$ and $(z_2,z_2^*)=(-\lambda, -\lambda^*)$ in formula \eqref{eq:section}, we observe that the values of the Lie bracket restricted to $\Lie B_\CC$ coincide with the values of the factor of automorphy of $i^*d^*\mathcal{B}$ and moreover in the definition of $Z'(1)$ it is enough to require $(z_2,z_2^*) $ in $\HH_1(B(\CC),\QQ)$. 
	 Hence $Z'(1)$ is the smallest sub-torus which contains the values of the factor of automorphy of $i^*d^*\mathcal{B}$ and so
	  quotienting the fibre of $i^*d^*\mathcal{B}$ by $Z'(1),$ we get the trivial torsor.
	 \end{proof}

	 Because of the equality
	 $ (V \times V^*)^*  {pr}_*i^*d^* \mathcal{B}  =  {pr}_*  (V \times V^*)^* i^*d^* \mathcal{B} ,$
	  the  trivialization $\Psi$ \eqref{eq:Psi} of the pull-back $ (V \times V^*)^* i^*d^* \mathcal{B}$ of the torsor $i^*d^* \mathcal{B}$ over $\ZZ \times \ZZ$ defines a trivialization 
	 \begin{equation}\label{eq:pr_*Psi}
	 {pr}_*\Psi : \ZZ \times \ZZ \longrightarrow  (V \times V^*)^*  {pr}_* i^*d^* \mathcal{B} 
	 \end{equation}
	  	  of the pull-back via $ V \times V^*$ of the torsor $  {pr}_*i^*d^* \mathcal{B}$. Denoting by $\pi: {pr}_*i^*d^* \mathcal{B} \twoheadrightarrow (X^\vee \otimes Y/Z')(1)$ the projection on the second factor, we can summarize what we have done in the following diagram:
	 \begin{equation}\label{diagram}
	 \begin{array}{ccccc}
	(X^\vee \otimes Y)(1) &\stackrel{pr}{\twoheadrightarrow}&(X^\vee \otimes Y/Z')(1)&=& (X^\vee \otimes Y/Z')(1)\\[1mm]
	\circlearrowright &&\;\longuparrow\rlap{$\pi$}&&\;\longuparrow\rlap{$\pi$}\\[2mm]
	  i^*d^*\mathcal{B} &  \longrightarrow &B\times (X^\vee \otimes Y/Z')(1) &\longleftarrow & \ZZ \times \ZZ\times (X^\vee \otimes Y/Z')(1)\\[1mm]
	\longdownarrow  &&\longdownarrow &&\qquad \longdownarrow \; \;  \longuparrow  {pr}_* \Psi \\[1mm]
  B &= &B& \stackrel{V \times V^*}{\longleftarrow} & \ZZ \times \ZZ \rlap{.}\\
	 \end{array}
	 \end{equation}
	 
	 By push-down the point $\Psi(1,1)$ defined in \eqref{R} furnishes a point 
	  \begin{equation}\label{R'}
	{pr}_*\Psi(1,1) \in  \big(  (V \times V^*)^* {pr}_* i^*d^*\mathcal{B}\big)_{1,1}
	 \end{equation}
which corresponds to the point 
$pr_* \tR$
 in the fibre of $pr_* i^*d^*\mathcal{B} =B\times (X^\vee \otimes Y/Z')(1)$ over the point $(P,Q) \in  B \subseteq A^n\times A^{*s}.$

 Let
 \begin{equation}\label{eq:Z}
 	Z
 \end{equation}
 be \textit{the smallest $\gal$-sub-module of $X^\vee\otimes Y$} containing $Z'$ and such that the sub-torus $(Z/Z')(1)$ of $(X^\vee \otimes Y /Z')(1)$ contains the point
  \begin{equation}\label{piR'}
 \pi (pr_* \tR).
 \end{equation}
By the commutative diagram \eqref{diagram} and by \eqref{R}, we have the equalities
$$\pi (pr_* \tR) =\pi \circ {pr}_* \Psi(1,1)  =\pi \Big( {pr}_* \big(\psi(x_i,y^\vee_k)\big)_{i= 1, \dots, n \atop k= 1, \dots, s }\Big) .$$

 \begin{remark}\label{remark:torsion}
 	The point $P$ of $A^n$ and the point $\pi (pr_* \tR)$ of $ (X^\vee \otimes Y/Z')(1)$ are simultaneously torsion if and only if the point $R$ of $G^n$ is torsion.
 \end{remark}
 
 \subsection{The unipotent radical of $\Lie \Galmot (M)$}\label{subsection:LieUR(M)}

	Since $\mathrm{Gal}(\oK/K)$ acts through a finite quotient on $X^\vee \otimes Y$, we can turn a supplement $W \subseteq X^\vee \otimes Y$ of $Z$ into a $\mathrm{Gal}(\oK/K)$-module.
Consider the push-down $Pr_*i^*d^* \mathcal{B}$ of the torsor $i^*d^* \mathcal{B}$ via the projection $Pr:(X^\vee \otimes Y)(1) \twoheadrightarrow (X^\vee \otimes Y/W)(1) \cong Z(1).$ 
 We still denote $\tR$ the point of 
 $Pr_*i^*d^* \mathcal{B}$  living over $(P,Q) \in B.$ 
 By construction the torsor $Pr_*i^*d^* \mathcal{B}$ is the \textit{smallest} $Z(1)$-torsor over $B$ (with respect to pull-backs and push-downs) containing the point $\tR$.
 By \cite[Th\'eor\`eme 0.1 and 3.7]{B03} or \cite[Definition 4.3]{J}
 the unipotent radical of the Lie algebra of the motivic Galois group $\Galmot(M)$ of the 1-motive $M$ is the  \textit{smallest} extension of $B$ by  $Z(1)$ containing the point $\tR$. Since any extension is a torsor, by minimality we have a $Z(1)$-equivariant inclusion
 \[ Pr_*i^*d^* \mathcal{B}  \hookrightarrow  \Lie \UR (M)\]
which sends $\tR$ in itself. By \cite[Chap. III, Th\'eor\`eme 1.4.5 \textit{(ii)}]{Giraud} this inclusion is in fact an isomorphism:
 \[ Pr_*i^*d^* \mathcal{B} \cong   \Lie \UR (M) = \Lie  \W_{-1} \big(\Galmot (M)\big) .\]

 The factors system defining the extension $ \Lie \UR (M)$ of $B$ by $Z(1)$ is the restriction to $B \times B$ of the map
\begin{align}\label{factorsystem}
	(A^n \times A^{*s}) \times (A^n \times A^{*s}) & \longrightarrow \GG_m^{ns} ,\\
	\nonumber \big( (Z_1,Z_1^*), (Z_2,Z_2^*) \big)& \mapsto \big(\mathcal{W}_A^{\cP} (Z_{1i},Z^*_{2k}) \big)_{i=1, \dots,n \atop k=1, \dots,s}
\end{align}
where $(Z_\ell,Z_\ell^*)\in A^n \times  A^{*s}$ for $\ell=1,2$.


\section{Study of the dimension of the unipotent radical of $\Lie \Galmot (M)$}\label{studyDimUR}

If we assume $K = \oK \subseteq \CC$,
by \cite[Lemma 3.5]{BPSS} the dimension of the unipotent radical $\UR (M)$ of $\Galmot(M)$ is
\begin{equation}\label{eq:dimUR}
	\dim_\QQ \UR (M) = 2 \dim B + \dim Z(1).
\end{equation}
Hence to study the dimension of the unipotent radical $\UR(M)$ is equivalent to study the dimensions of the group varieties $B, Z'(1)$ and $(Z/Z')(1)$ introduced in Subsection \ref{subsection:BZ}.

In this section we work over $K$ and we will mention when we need $K$ algebraically closed in $\CC$ as above.

\subsection{Degeneracies of the abelian variety $B$}\label{subsection:B} The maximal dimension of $B \subseteq A^n \times A^{*s} $ is 
 $(n+s) g$
where $g$ is the dimension of the abelian variety $A.$
The dimension of $B$ decreases if there are relations between the  $P_i$'s and $Q_k$'s given by endomorphisms of $A^n \times A^{*s}$.

\begin{example}\label{ex:dimA} Let $A$ be a simple, $g$-dimensional abelian variety and set $F=\End( A) \otimes_\ZZ \QQ$ its (possibly skew) field of endomorphisms.
Let $\omega_1, \dots , \omega_g $ be invariant differentials of the first kind on $A$ spanning the $g$-dimensional vector space $\HH^0(A, \Omega^1_A)$ of holomorphic differentials and let $\lambda_1,\dots,\lambda_{2g}$ be a base of $\HH_1(A(\CC),\ZZ).$ 
The periods of $A$ are the $2 g^2$ complex numbers 
$\omega_{lj} =\int_z^{\lambda_j +z}\omega_l $ for $l=1, \dots ,g$ and $j=1, \dots, 2g$. Let $p_i =(\int_O^{P_i}\omega_1, \dots , \int_O^{P_i}\omega_g) $ and 
$q_k =(\int_O^{Q_k}\omega_1, \dots ,  \int_O^{Q_k}\omega_g)$ be abelian logarithms of the points $P_i$'s and $Q_k$'s for $i=1, \dots ,n$ and $k=1, \dots,s$ (here $O$ is the neutral element of the abelian variety $A$ and we identify $q_k$ with $\phi_H(q_k)$). Recall that the abelian logarithms depend on the paths up to $F$-linear combinations of the $\omega_{lj}$. Then the dimension of $B$ is $g$ times the dimension of the $F$-vector space generated modulo $\sum_{j=1}^{2g}F(\omega_{1j},\dots,\omega_{gj})$ by the points $p_i$'s and $q_k$'s:
$$ \dim B = g \cdot \dim_F \Bigg( \Big(\sum_{i=1, \dots, n \atop k=1, \dots , s} F p_i + F q_k +\sum_{j=1}^{2g}F(\omega_{1j},\dots,\omega_{gj}) \Big) \Big/ \sum_{j=1}^{2g}F(\omega_{1j},\dots,\omega_{gj}) \Bigg).$$ 
Indeed $B$ is isogenous to $A^N $ with $N$ the dimension of the above $F$-vector space.
\end{example}

If the abelian variety $A$ isn't simple, see \cite[Lemma 4.3]{BPSS} for the computation of the dimension of $B$.

\subsection{Degeneracies of the torus $Z'(1)$}\label{subsectionZ'(1)} The sub-torus $Z'(1)$ of $(X^\vee \otimes Y)(1)$ has maximal dimension 
$ns$. As we will see now this dimension is governed by the abelian variety $B$.
 The inclusion $j:B \hookrightarrow A^n\times {A^*}^s$ induces group morphisms
\begin{align}
\nonumber j :B &\longrightarrow  A^n\times {A^*}^s\\
\nonumber b &\longmapsto \big(\gamma_1(b),\dots,\gamma_{n}(b),\gamma_{1}^*(b),\dots,\gamma_{s}^*(b)\big)
\end{align}
where  $\gamma_i\in\mathrm{Hom}_{\mathbb Q}(B,A) := \mathrm{Hom}(B,A) \otimes_\ZZ \QQ $ ({\it resp.} $\gamma^*_k \in\mathrm{Hom}_{\mathbb Q}(B,A^*)$) is the composition of $j$ with the projection on the $i$-th factor $A$ of $A^n$ ({\it resp.} on the $k$-th factor $A^*$ of $ {A^*}^s$) for $i=1,\dots,n$ ({\it resp.} $k=1,\dots,s$).
We denote with an upper-index ${}^t$ the transpose of a group morphism. 
Set 
\begin{equation}\label{eq:beta}
	\beta_{i,k}: =\gamma_i^t\circ\gamma_k^* \in \mathrm{Hom}_{\mathbb Q}(B,B^*).
	\end{equation}
 Observe that $\beta_{i,k}^t=\gamma_k^{*t} \circ \gamma_i  \in \mathrm{Hom}_{\mathbb Q}(B,B^*)$ and $j^t$ is the natural projection ${A^*}^n\times A^s \to B^* = {A^*}^n\times A^s/ C^* $ with $C=A^n\times{A^*}^s/ j(B)$. The homomorphism $\beta_{i,k}$ ({\it resp.} $\beta_{i,k}^t$) can be written as $ j^t \circ \ell_{i,k} \circ j $ where $\ell_{i,k} : A^n \times A^{*s} \to A^{*n} \times A^s$ is the $\QQ$-homomorphism defined by the $(n+s)\times(n+s)$ matrix with a $1$ in position $(n+k,i)$ ({\it resp.} $(i,n+k)$) and zeros elsewhere. In particular the homomorphism $\beta_{i,k}+\beta_{i,k}^t$ sends a point $b\in B$ with $j(b) = (Z,Z^*)=(Z_1,\dots,Z_n,Z^*_1,\dots,Z^*_s)$ onto the point $(0, \dots, 0, Z^*_k, 0, \dots, 0, Z_i, 0, \dots, 0) \in {A^*}^n\times A^s$ modulo $C^*$, where the point $Z^*_k$ is in the $i$-th factor $A^*$ and $Z_i$ is in the $k$-th factor $A$.  As usual 
 denote by $x_1, \dots, x_{n}$ (\textit{resp}. $y^\vee_1, \dots , y^\vee_{s}$) the basis of the character group $X$ (\textit{resp.} $Y^\vee$) giving the identification of $X$ with $\ZZ^n$ (\textit{resp}. $Y^\vee$ with  $\ZZ^s$) and set
 \[v(x_i)=P_i \in A(\oK), \qquad v^*(y^\vee_k) =Q_k \in A^*(\oK)\]
 for $ i=1, \dots,n $ and $ k=1, \dots,s. $

\begin{theorem}\label{teo:dimZ'(1)}
	The space of equations of $Z'(1)$ is the kernel of the surjective map
	\begin{equation}\label{applicentrale}
		\begin{matrix}
			f:(X\otimes Y^\vee)_\QQ &\twoheadrightarrow &\sum_{i,k}\QQ (\beta_{i,k} + \beta_{i,k}^t) \subset \mathrm{Hom}_\QQ (B,B^*)\\
			\hfill x_i\otimes y_k^\vee &\mapsto & \beta_{i,k} + \beta_{i,k}^t \hfill
		\end{matrix}
	\end{equation}
	In particular
the dimension of the torus $Z'(1)$ is equal to the dimension of the $\QQ$-vector space generated in $\mathrm{Hom}_\QQ(B,B^*)$ by the elements $\beta_{i,k}+\beta_{i,k}^t$ for $i=1,\dots,n$ and $k=1,\dots,s$.
\end{theorem}

\begin{proof} 
If $b,b'$ are two arbitrary points of $B$, using definition \eqref{[--]} plus (4) and (1) of Lemma \ref{propieteW_A} we compute the component $(i,k)$ of the Lie bracket $[j(b),j(b')]$:
\begin{align}\label{eq:W(b,b')}
 \big[j(b),j(b')\big]_{i,k} &= 	\mathcal{W}_A\big(\gamma_i(b),\gamma^*_k(b')) \wedge	\mathcal{W}_{A^*}(\gamma^*_k(b),\gamma_i(b')\big) \\
\nonumber &= \mathcal{W}_B\big(b, \gamma_i^t \circ \gamma^*_k(b')) \wedge \mathcal{W}_{B}(b,\gamma^{*t}_k \circ\gamma_i(b')\big) \\
\nonumber &=  \mathcal{W}_B\big(b,(\gamma_{i}^t\circ\gamma_{k}^*+\gamma_{k}^{*t}\circ\gamma_{i})(b')\big)\\
\nonumber &=  \mathcal{W}_B\big(b,(\beta_{i,k}+\beta_{i,k}^t)(b')\big).
\end{align}
 In particular, for $\lambda_{i,k}\in\mathbb Z$,
\begin{align*}
\prod_{i=1}^n\prod_{k=1}^s \big[j(b),j(b')\big]_{i,k}^{\lambda_{i,k}} 
&= \prod_{i=1}^n\prod_{k=1}^s  	\mathcal{W}_B\big(b,(\beta_{i,k}+\beta_{i,k}^t)(b')\big)^{\lambda_{i,k}} \\
&=  	\mathcal{W}_{B}\Big(b,\sum_{i=1}^n\sum_{k=1}^s\lambda_{i,k}(\beta_{i,k}+\beta_{i,k}^t)(b')\Big)
.\end{align*}
Thus, multiplicative relations
\begin{equation}\label{eqZ'(1)}
\prod_{i=1}^n\prod_{k=1}^s\big[j(b),j(b')\big]_{i,k}^{\lambda_{i,k}}=1
\end{equation}
for $b$ and $b'$ running over $B$, are equivalent to relations
$$ 	\mathcal{W}_{B}\Big(b,\sum_{i=1}^n\sum_{k=1}^s\lambda_{i,k}(\beta_{i,k}+\beta_{i,k}^t)(b')\Big) = 1$$
for $b$ and $b'$ running over $B$. By Lemma \ref{propieteW_A} (3) it is further equivalent to
\begin{equation}\label{eq:crochet=0}
\sum_{i=1}^n\sum_{k=1}^s\lambda_{i,k}(\beta_{i,k}+\beta_{i,k}^t)=0.
\end{equation}
Hence the equations \eqref{eqZ'(1)} of the torus $Z'(1)$ coincide with the kernel of $f$.
\end{proof}

\begin{remark} 
	\begin{enumerate}
		\item The condition \eqref{eq:crochet=0} is also expressed as $\ell \circ j(B) \subset C^*$, where $\ell:A^n\times {A^*}^s \to {A^*}^n\times A^s$ is the $\mathbb Q$-homomorphism defined by the matrix $\Big(\begin{matrix}0 &L\\ L^t &0\end{matrix}\Big)$ with $L:=\left(\lambda_{i,k}\right)_{i =1, \dots ,n \atop k  =1, \dots ,s}$.
		\item Consider the restriction $[\cdot,\cdot] \circ (j \otimes j): B \otimes B \longrightarrow Z'(1)$ to $B \otimes B$ of the Lie bracket \eqref{eq:bracket}. According to the equalities \eqref{eq:W(b,b')} and Lemma \ref{propieteW_A} (3), the set of points of $B$ which are orthogonal to all points of $B$ with respect to this restriction is 
		\[B_\tors + \bigcap_{i =1, \dots ,n \atop k  =1, \dots ,s} \ker (\beta_{i,k}+\beta_{i,k}^t). \]
	\end{enumerate}
\end{remark}

\begin{remark} \begin{enumerate}
			\item  If $B=0$, we have that for all $i$ and $k,$ the group morphisms $\beta_{i,k} \in \mathrm{Hom}_{\mathbb Q}(B,B^*)$ are zero, and hence $Z'(1)=\{1\}.$
		\item If $n=s=1$, it exists a non trivial relation \eqref{eq:crochet=0} (i.e. $\lambda_{i,k}$ not all zero) if and only if $\beta_{1,1} = \gamma_1^t \circ \gamma_1^*$ is antisymmetric (i.e. $\beta_{1,1}+\beta^t_{1,1}=0$), if and only if $Z'(1)$ is trivial. If we denote by $\sym_{A \times A^*} $ the group morphism from $A\times A^*$ to $A^*\times A$ which exchanges the factors, by \cite[Proposition 1]{Br94} and thereafter, the vanishing of $j^t\circ \sym_{A \times A^*} \circ j$ on $B$ is equivalent to the fact that the abelian sub-variety $B$ is isotropic in $A\times A^*$, i.e. the restriction to $B$ of the Poincar\'e biextension $\mathcal P$ on $A \times A^*$ is algebraically equivalent to zero. But, $\sym_{A \times A^*} \circ j$ is the morphism $(\gamma^*_{1},\gamma_{1})$ and $j^t=\gamma_{1}^t + \gamma_{1}^{*t}$ in the notations above. Thus, $j^t\circ \sym_{A \times A^*} \circ j = \gamma_{1}^t\circ\gamma_{1}^* + {\gamma_{1}^{*t}}\circ\gamma_{1}$ is $\beta_{1,1}+\beta_{1,1}^t$.
	In particular $Z'(1)=\{1\}$ if and only if $B$ is isotropic in $A \times A^*.$
 \item If $A$ is an elliptic curve not CM and $n=s=1$, the dimension of $Z'(1)$ is 1, since there are no antisymmetric group morphism.
\end{enumerate}
\end{remark}

As a consequence of the above Theorem (see also \cite[Proposition 1, iii bis)]{Br94}) and of Example \ref{ex:dimA} we have that  

\begin{corollary}\label{cor:DimSemi-ellipticCase}
	Suppose that $n=s=1$ and that $A$ is a simple $g$-dimensional abelian variety. With the above notations we have that
	\begin{enumerate}
		\item if $P$ and $Q$ are torsion points, then $\dim B= \dim Z'(1)=0 $;
		\item if $P$ is a torsion point but not $Q$, or if $Q$ is a torsion point but not $P$, $\dim B= g$ and $\dim Z'(1)=0$;
		\item if $P$ and $Q$ are not torsion and there exists a group morphism $f \in \Hom_{\QQ} (A,A^*) $ (\textit{resp.} $f \in \Hom_{\QQ} (A^*,A) $) such that $f(P)=Q$ (\textit{resp.} $f(Q)=P$), then $\dim B= g$ and 
		\begin{enumerate}
			\item if $f$ is antisymmetric,  $\dim Z'(1)=0$;
			\item if $f$ is not antisymmetric,  $\dim Z'(1)=1$;
		\end{enumerate}
		\item  if $P$ and $Q$ are $\Hom_{\QQ} (A,A^*)$-linearly independent, then $\dim B= 2g$ and $\dim Z'(1)=1$.
	\end{enumerate}
\end{corollary}

\begin{corollary}\label{cor:maxdim}
If $B$ is maximal, i.e. $B=A^n\times A^{*s}$, the dimension of $Z'(1)$ is maximal, i.e. $\dim Z(1) = \dim Z'(1) = ns$.
In particular, assuming $K=\oK$ the dimension of the unipotent radical $\UR(M)$ of the motivic Galois group $\Galmot(M)$ is maximal, that is
\[ \dim_\QQ \UR(M) = 2 (n+s) g + ns\]
where $g$ is the dimension of the abelian variety $A$.
\end{corollary}

\begin{proof} We will use the notation of Theorem \ref{teo:dimZ'(1)}. Since $B=A^n\times A^{*s},$
	 we may take $\gamma_{i}\in\mathrm{Hom}_{\mathbb Q}(B,A)$ as the projection of $B$ onto the $i$-th factor $A$ of $A^n\times A^{*s}$ for $i=1, \dots, n$, and $\gamma_k^*\in\mathrm{Hom}_{\mathbb Q}(B,A^*)$ as the projection of $B$ onto the $k$-th factor $A^*$ of $A^n\times A^{*s}$ for $k=1, \dots, s$. 
	 The transpose $\gamma_{i}^t: A^* \to B^* = A^{*n}\times A^{s}$ is the inclusion of $A^*$ into the $i$-th factor $A^*$ of $A^{*n}\times A^{s},$ and 
	 $\gamma_{k}^{*t}: A \to B^* = A^{*n}\times A^{s}$ is the inclusion of $A$ into the $k$-th factor $A$ of $A^{*n}\times A^{s}.$
	 Therefore $\beta_{i,k}=\gamma_i^t\circ\gamma_k^*$ is the identity between the $k$-th factor $A^*$ of $B$ and the $i$-th factor $A^*$ of $B^*=A^{*n}\times A^{s}$ and zero on the other factors. Similarly $\beta_{i,k}^t = \gamma_k^{*t} \circ \gamma_i$ is the identity between the $i$-th factor $A$ of $B$ and the $k$-th factor $A$ of $B^*=A^{*n}\times A^{s}$ and zero on the other factors.
	 For any $x$ in the $k$-th factor $A^*$ of $B$, the component of $(\beta_{i,k}+\beta_{i,k}^t) (x)$ in the $i$-th factor $A^*$ of $B^*=A^{*n}\times A^{s}$ is just $x$.
	 
	  Assume there is a relation with coefficients in $\QQ$
	 \[\sum_{i=1}^n\sum_{k=1}^s\lambda_{i,k}(\beta_{i,k}+\beta_{i,k}^t) = 0.\]
	 Evaluating the left hand side in a point $x$ as above, its component in the $i$-th factor $A^*$ of $B^*=A^{*n}\times A^{s}$ is just $\lambda_{i,k} \, x .$ By the above relation this implies $\lambda_{i,k}=0$. Hence the $\beta_{i,k}+\beta_{i,k}^t$ are linearly independent over $\mathbb Q$ and by Theorem \ref{teo:dimZ'(1)}, the dimension of $Z'(1)$ is therefore maximal, i.e. $\dim Z(1) = \dim Z'(1) = ns$. Using equality \eqref{eq:dimUR} we conclude.
\end{proof}

\subsection{Degeneracies of the torus $(Z/Z')(1)$}\label{subsection:Z/Z'(1)} As seen in Subsection \ref{subsection:BZ}, $Z(1)$ is the smallest sub-torus of $(X^\vee \otimes Y)(1)$ containing $Z'(1)$ and such that $(Z/Z')(1)$ contains the point $  \pi (pr_*\tR)$ constructed in \eqref{piR'}. Since $\dim (Z/Z')(1) = \dim Z(1) -\dim Z'(1)$, the maximal dimension of the torus $(Z/Z')(1)$ is $ns $
and clearly it can be achieved only if the torus $Z'(1)$ is trivial. 
In general the dimension of $(Z/Z')(1)$ is governed by the $K$-rational point $ \pi (pr_*\tR)$, as we will see now.

Let again $x_1, \dots, x_{n}$ (\textit{resp}. $y^\vee_1, \dots , y^\vee_{s}$) the basis of the character group $X$ (\textit{resp.} $Y^\vee$) giving the identification of $X$ with $\ZZ^n$ (\textit{resp}. $Y^\vee$ with  $\ZZ^s$) over $\oK$ and set
\[v(x_i)=P_i \in A(\oK), \qquad v^*(y^\vee_k) =Q_k \in A^*(\oK)\]
for $ i=1, \dots,n $ and $ k=1, \dots,s. $ This gives an isomorphism over $\oK$
\begin{align*}
	(X^\vee\otimes Y)(1) \simeq \underline{\mathrm{Hom}}(X\otimes Y^\vee,\mathbb G_m) &\longrightarrow \mathbb G_m^{ns}\\
	\big(x_i\otimes y_k^\vee \mapsto t_{i,k}\big) &\longmapsto \big(t_{i,k}\big)
	.\end{align*}
Starting from the trivialization $\psi:X \times Y^\vee \to (v \times v^*)^* \cP$ underlying the 1-motive $M$, in \eqref{R} we have constructed a point $\tR \in (i^*d^* \mathcal{B})_{P,Q}$ (corresponding to the point $\Psi(1,1) \in \big(( V \times V^*)^* i^*d^* \mathcal{B} \big)_{1,1}),$ which has furnished by push-down the point $pr_* \tR \in (pr_*i^*d^* \mathcal{B})_{P,Q} $ (corresponding to the point $pr_*\Psi(1,1) \in \big(( V \times V^*)^* pr_* i^*d^* \mathcal{B} \big)_{1,1})$ . The torus $(Z/Z')(1)$
is the smallest sub-torus of $(X^\vee \otimes Y/Z')(1)$ which contains the point 
\[ \pi (pr_*\tR)= \pi \Big( {pr}_* \big(\psi(x_i,y^\vee_k)\big)_{i= 1, \dots, n \atop k= 1, \dots, s } \Big)\]
where $\pi: pr_* i^*d^* \mathcal{B} \to (X^\vee \otimes Y/Z')(1)$ denotes the projection on the second factor of the trivial torsor $ pr_* i^*d^* \mathcal{B}$ (see Lemma \ref{lemma:pr*i*d*B-trivial}).

\medskip

The pull-back $\exp_B^* i^*d^* \mathcal{B}$ of the $X^\vee \otimes Y(1)$-torsor $i^*d^* \mathcal{B}$, introduced in \eqref{i^*d^*b}, is the trivial $X^\vee \otimes Y(1)$-torsor 
$\Lie B_\CC \times ( X^\vee \otimes Y)(1).$
Consider two logarithms $(p,q)$ and $(p',q')$  of the point $(P,Q) \in B$. The pull-backs 
 $(p,q,t)$ and $(p',q',t')$
  of the point $\tR \in  (i^*d^* \mathcal{B})_{P,Q} $ in the two fibres $(\exp_B^* i^*d^* \mathcal{B})_{p,q}$ and $(\exp_B^* i^*d^* \mathcal{B})_{p',q'}$  are such that $t (t')^{-1}$ is a value of the factor of automorphy of $i^*d^* \mathcal{B}$, which belongs to the torus $Z'(1)$ as we have showed in Lemma \ref{lemma:pr*i*d*B-trivial}. 
 The pull-back of the point  $ pr_* \tR \in  (pr_*i^*d^* \mathcal{B})_{P,Q} $ is the  point
 \begin{equation}\label{eq:exp_B^*}
 \big(p,q,t \; \mathrm{mod} \, Z'(1) \big) \in (\exp_B^*pr_* i^*d^* \mathcal{B})_{p,q}=  \Big(\Lie B_\CC \times  (X^\vee \otimes Y/Z')(1)\Big)_{p,q}
 \end{equation} 
 and therefore it doesn't depend on the choice of the logarithms of $(P,Q)$. Hence
 \begin{equation*}
 	 pr_* \tR = (P,Q,t \; \mathrm{mod} \, Z'(1) ) \in B \times (X^\vee\otimes Y/Z')(1)
 \end{equation*}
  and in particular
 \begin{equation*}
 \pi (pr_* \tR) = \pi \Big( {pr}_* \big(\psi(x_i,y^\vee_k)\big)_{i= 1, \dots, n \atop k= 1, \dots, s } \Big) =  t \; \mathrm{mod} \, Z'(1) 
 \in (X^\vee\otimes Y/Z')(1).
 \end{equation*}

  Now pick a point $t=(t_{i,k})\in (X^\vee\otimes Y)(1) \simeq \GG_m^{ns}$, possibly defined over a finite extension of $K$, which projects via $(X^\vee\otimes Y)(1) \to (X^\vee\otimes Y/Z')(1) $ onto $  \pi (pr_* \tR)= \pi \Big( pr_*\big(\psi(x_i, y_k^\vee)\big)_{i,k}\Big)$. By Theorem \ref{teo:dimZ'(1)} the space ${Z'}^\perp_\QQ := {Z'}^\perp\otimes_{\ZZ}\QQ $ of equations of $Z'(1)$ is the kernel of the surjective map
  \[	\begin{matrix}
  	f:(X\otimes Y^\vee)_\QQ &\twoheadrightarrow &\sum_{i,k}\QQ (\beta_{i,k} + \beta_{i,k}^t) \subset \mathrm{Hom}_\QQ (B,B^*)\\
  	\hfill x_i\otimes y_k^\vee &\mapsto & \beta_{i,k} + \beta_{i,k}^t  \; . \hfill
  \end{matrix}
  \]
  The point $t$ gives rise, by restriction, to a point $R_0\in({Z'}^\perp)^\vee(1)$:
	\begin{equation}\label{pointpsi}
		\begin{matrix}
			{Z'}^\perp &\longrightarrow &\mathbb G_m\\[1mm]
			\sum_{i,k}\alpha_{i,k}x_i\otimes y_k^\vee &\longmapsto &\prod_{i,k}t_{i,k}^{\alpha_{i,k}}
		\end{matrix}
	\end{equation}
	which does not depend on the above choice of $t=(t_{i,k})$, since points in $Z'(1)$ take the value $1$ on ${Z'}^\perp$. By minimality of the torus $Z(1)$ defined in \eqref{eq:Z}, the space ${Z}^\perp_\QQ := {Z}^\perp\otimes_{\ZZ}\QQ $ of equations of $Z(1)$ is the kernel the above map \eqref{pointpsi}.
	 The dual ${Z'}^\vee \subset X \otimes Y^\vee $ of $Z'$  identifies with the quotient $(X\otimes Y^\vee)/{Z'}^\perp$ through the exact sequence
	of $\mathrm{Gal}(\oK/K)$-modules
	$${Z'}^\perp \hookrightarrow X\otimes Y^\vee \stackrel{f}{\twoheadrightarrow} {Z'}^\vee.$$
	Since $\mathrm{Gal}(\oK/K)$ acts through a finite quotient on these modules, we can turn a supplement in $ (X\otimes Y^\vee)_\QQ$ of the $\QQ$-vector space ${Z'}^\perp_{\mathbb Q}$ into a $\mathrm{Gal}(\oK/K)$-module, getting a $\mathrm{Gal}(\oK/K)$-invariant decomposition
	$$(X\otimes Y^\vee)_{\mathbb Q} = {Z'}^\perp_{\mathbb Q}\oplus {Z'}^\vee_{\mathbb Q} \simeq \ker(f)\oplus\mathrm{im}(f),$$
	with $\ker(f)$ the space of relations between the $ (\beta_{i,k} + \beta_{i,k}^t)$'s in $\mathrm{Hom}_\QQ(B,B^*)$ and $\mathrm{im}(f) $ the space $ \sum_{i,k}\mathbb Q (\beta_{i,k} + \beta_{i,k}^t) \simeq {Z'}^\vee_{\mathbb Q}$.
	Since $f$ is surjective, its dual map $f^\vee$ is injective, and so we have the dual exact sequence of $ \mathrm{Gal}(\oK/K)$-modules
	$$Z' \stackrel{f^\vee}{\hookrightarrow} X^\vee\otimes Y \twoheadrightarrow ({Z'}^\perp)^\vee$$
	which, as above, implies a $\mathrm{Gal}(\oK/K)$-invariant decomposition of the $\mathbb Q$-vector space generated by the co-characters of the torus $(X^\vee\otimes Y)(1)$:
	$$(X^\vee\otimes Y)_{\mathbb Q} = Z'_{\mathbb Q}\oplus ({Z'}^\perp)^\vee_{\mathbb Q} \simeq \mathrm{im}(f^\vee) \oplus \mathrm{coker}(f^\vee).$$
	The latter shows that the torus $(X^\vee\otimes Y)(1)$ is the direct product of its sub-torii $Z'(1)$ and $({Z'}^\perp)^\vee(1)$:
	\[(X^\vee\otimes Y)(1) = ({Z'}^\perp)^\vee(1) \times  Z'(1)  \]
	and, by construction, $R_0$ which is the projection of the point $t \in (X^\vee\otimes Y)(1)$ into the first factor, corresponds to the point 
	$\pi\Big( pr_*\big(\psi(x_i, y_k^\vee)\big)_{i,k}\Big)$
	 via the isomorphism $({Z'}^\perp)^\vee(1) \simeq (X^\vee\otimes Y/Z')(1).$
 In particular  $Z'(1) \cap ({Z'}^\perp)^\vee(1) =\{1\}.$
	 The torus $({Z'}^\perp)^\vee(1)$ is the intersection in $(X^\vee\otimes Y)(1)$ of the kernels of the characters in ${Z'}^\vee$.

	Define $Z''(1)$ as the smallest sub-torus of $({Z'}^\perp)^\vee(1)$ which contains the point $R_0$ defined in \eqref{pointpsi} or a multiple of it that we still denote $R_0$: in particular $Z''^\vee$ is contained in ${Z'}^\perp$ and the restriction of the point $R_0$ to $Z''^\vee$ has trivial kernel by minimality of $Z''(1)$.
	The torus $Z'(1)Z''(1)$  is the smallest sub-torus of $(X^\vee\otimes Y)(1)$ containing $Z'(1)$ and the point $R_0$. Therefore by \eqref{eq:Z}
	\begin{equation}\label{Z''}
		Z(1)=Z'(1)Z''(1) 
	\end{equation}
	  and in particular the dimension of $(Z/Z')(1) $ is equal to the dimension of $ Z''(1)$.  Since the restriction of the point $R_0$ to $Z''^\vee$ has trivial kernel, $Z''^\vee_\QQ$ is isomorphic to the image of \eqref{pointpsi} tensorized with $\QQ$, which in turn is isomorphic to the $\QQ$-vector subspace of  $\CC/2\mathrm i\pi\QQ$ generated by the logarithms of the coordinates of $R_0$ modulo $2\mathrm i\pi\QQ.$ We have then proved the following

	\begin{theorem}\label{teo:dimZ/Z'}
		The space ${Z}^\perp_\QQ$ of equations of $Z(1)$ is the kernel of the map \eqref{pointpsi}. Moreover the $\QQ$-vector subspace
	 $(Z/Z')_\QQ$ is isomorphic to the $\QQ$-vector subspace of  $\CC/2\mathrm i\pi\QQ$ generated by the logarithms $\sum_{i,k}\alpha_{i,k}\log(t_{i,k})$, with $\sum_{i,k}\alpha_{i,k} x_i\otimes y_k^\vee$ running over ${Z'}^\perp$ and $(t_{i,k})$ any point of $(X^\vee\otimes Y)(1)$ projecting onto $ \pi \Big( pr_*\big(\psi(x_i, y_k^\vee)\big)_{i,k} \Big)$.
	\end{theorem}

\begin{corollary}\label{prop:dimZ/Z'}
	If $Z'(1)=\{1\},$ the dimension of the torus $Z(1)$ is equal to the dimension of the  $\QQ$-sub-vector space of $\CC / 2 i \pi \QQ$ generated by the logarithms of the components of $\pi \Big(\big(\psi(x_i,y^\vee_k)\big)_{i= 1, \dots, n \atop k= 1, \dots, s }\Big)$, where  $x_1, \dots, x_{n}$ and $y^\vee_1, \dots , y^\vee_{s}$ are generators of the character groups $X$ and $Y^\vee$ respectively.
\end{corollary}


\section{Periods matrices and motivic Galois action}\label{lastSection}

\subsection{Periods of abelian varieties}\label{periodsonabelianvarieties} Let $A$ be an abelian variety defined over $ K\subseteq \CC.$ 
Recall the notations of Subsection \ref{abelianvariety}.
  We fix an ordered base $\lambda_1,\dots,\lambda_{2g}$ of the lattice $\Lambda\simeq \mathrm H_1(A(\CC),\mathbb Z)$ which is symplectic for the alternating form $\mathrm{Im}(H)$. 
  The invariant differential forms $\omega_1=dz_1,\dots,\omega_g=dz_g$ form a base of differentials of the first kind on $A$. The invariant differential forms $\eta_1,\dots,\eta_g$, defined by the polarisation, give a base of differentials of the second kind modulo holomorphic and exact ones on $A$. We use the ordered base $\omega_1,\dots,\omega_g,\eta_1,\dots,\eta_g$ of $\mathrm H^1_{\mathrm{dR}}(A)$. Then, the base of $\mathrm H_1(A^*(\CC),\mathbb Z)$ dual to $\lambda_1,\dots,\lambda_{2g}$ with respect to $\mathrm{Im}(H)$ is $-\lambda_{g+1},\dots,-\lambda_{2g},\lambda_1,\dots,\lambda_g$ and the dual base of $\mathrm H^1_{\mathrm{dR}}(A^*)$ is the ordered base $-\eta_1,\dots,-\eta_g,\omega_1,\dots,\omega_g$. Define the $g\times g$ matrices
$$\begin{matrix}\Omega_1 = \left(\begin{matrix}\int_z^{\lambda_i+z}\omega_j\end{matrix}\right)_{\stackrel{i=1,\dots,g}{\scriptscriptstyle j=1,\dots,g}} \hfill&{\cyr I}_1 = \left(\begin{matrix}\int_z^{\lambda_i+z}\eta_j\end{matrix}\right)_{\stackrel{i=1,\dots,g}{\scriptscriptstyle j=1,\dots,g}}\hfill\\
\Omega_2 = \left(\begin{matrix}\int_z^{\lambda_{g+i}+z}\omega_j\end{matrix}\right)_{\stackrel{i=1,\dots,g\hfill}{\scriptscriptstyle j=1,\dots,g\hfill}} \hfill &\cyr I_2 = \left(\begin{matrix}\int_z^{\lambda_{g+i}+z}\eta_j\end{matrix}\right)_{\stackrel{i=1,\dots,g\hfill}{\scriptscriptstyle j=1,\dots,g\hfill}}\end{matrix}$$
so that the periods matrices of $A$ and $A^*$ with respect to the above chosen bases are
\begin{equation}\label{periodsmatricesA}
 \Pi_A = \left(\begin{matrix}\Omega_1 &\cyr I_1\\\Omega_2 &\cyr I_2\end{matrix}\right) \qquad \mathrm{and} \qquad \Pi_{A^*} = \left(\begin{matrix}\cyr I_2 &-\Omega_2\\ -\cyr I_1 &\Omega_1\end{matrix}\right)
\end{equation}
Thanks to Lemma~\ref{lemcroixperiodquasi} we have the following relations
$$\left(\begin{matrix}\Omega_1 &\cyr I_1\\
\Omega_2 &\cyr I_2\end{matrix}\right)\left(\begin{matrix} -\cyr I_1^t &-\cyr I_2^t\\ \Omega_1^t &\Omega_2^t\end{matrix}\right) = \left(\begin{matrix} 
\cyr I_1\Omega_1^t - \Omega_1\cyr I_1^t
&\cyr I_1\Omega_2^t - \Omega_1\cyr I_2^t\\
\cyr I_2\Omega_1^t - \Omega_2\cyr I_1^t
&\cyr I_2\Omega_2^t - \Omega_2\cyr I_2^t
\end{matrix}\right) = -2\mathrm i\pi \left(\begin{matrix}0 &\mathrm{Id}_g\\ -\mathrm{Id}_g &0\end{matrix}\right) = -2\mathrm i\pi J
$$
where $J$ is the matrix of the alternating form $\mathrm{Im}(H)$ associated to the principal polarization $H$ in the symplectic base of $\Lambda$ (see \cite[\S 3.1, page 46]{BL} and \cite[\S 23.2.2, page 204]{A04}). Two little computations then show that
$	\Pi_{A^*} =-J\Pi_AJ $ and $
\Pi_{A^*}  = (2\mathrm  i\pi\Pi_A^{-1})^t .$
 In particular, $\det(\Pi_{A^*})\det(\Pi_A)=(2\mathrm i\pi)^{2g}$ and $\det(\Pi_A)=\det(\Pi_{A^*})$ and thus $\det(\Pi_A)=\pm(2\mathrm i\pi)^g$, which reduces to Legendre relation if $g=1$.

	
	\medskip
	
	According to \cite[\S 1.11]{D90}, denote by $\underline{\mathrm{Isom}}^\otimes(  \HH_\dR, \HH_\HH ) $ the functor over the $K$-schemes, which associates to any $f: T \to \mathrm{Spec}(K)$ the set of isomorphisms of $\otimes$-functors  $f^* \HH_\dR \to f^* \HH_\HH $ , where $  \HH_\dR$ and $ \HH_\HH$ are respectively the fibre functors ``De Rham realization" and ``Hodge realization".
	 By \cite[Proposition 6.6 and \S 2.7)]{D90}, $\underline{\mathrm{Isom}}^\otimes(  \HH_\dR, \HH_\HH ) $ is representable by an affine $K$-scheme.
	 Denote by $<A>^\otimes$ the tannakian sub-category generated by the abelian variety $A$ in a suitable tannakian category of motives. The scheme 
	 $\underline{\mathrm{Isom}}^\otimes(  {\HH_\dR}_{|<A>^\otimes }, {\HH_\HH}_{|<A>^\otimes } ) $ can be identified as a sub-scheme of the scheme of isomorphisms of the two vector spaces  $ \HH_\dR (A) \otimes_K \CC$ and $\HH_\HH (A) \otimes_\QQ \CC.$ Following \cite[page 73, 82-83]{A04}, via this identification, the \textit{periods isomorphism} $ \varpi_A : \HH_\dR (A) \otimes_K \CC \rightarrow \HH_\HH (A) \otimes_\QQ \CC$ given by the integration of differentials forms corresponds to a  
	  $\CC$-rational point in the scheme
	 $
	 \underline{\mathrm{Isom}}^\otimes(  {\HH_\dR}_{|<A>^\otimes }, {\HH_\HH}_{|<A>^\otimes } )(\CC) .$ After choosing bases, this periods isomorphism $ \varpi_A $ gives a matrix called a \textit{periods matrix} of $A$, whose coefficients are called \textit{periods} of $A$.

	  Since $ \HH_\dR (A) = \oplus_{r=0}^{2g}\HH^r_{\dR}(A) $ 
	and $\HH_\HH (A) = \oplus_{r=0}^{2g}\HH_r(A(\CC),\mathbb Q) $ we define the \textit{$r$-periods of $A$} (for $r=0, \dots,2g$) as the coefficients of the matrix which represents (with respect to $K$-bases) the isomorphism between the cohomology groups $\HH_\dR^r(A)$ and $\HH_r(A(\CC),\mathbb Q) \otimes_\QQ K$ which is given by the integration of differentials forms. In particular the coefficients of the matrix $\Pi_A$  in \eqref{periodsmatricesA} are the 1-periods of $A$.

	\begin{example}\label{legendreRelation}
		In the case of an elliptic curve, the only 0-period is 1, the four 1-periods are the classical periods $\omega_1, \omega_2, \eta_1$ and $\eta_2$, and the only 2-period is $ 2 \mathrm i \pi$ (see \cite[7.1.6.1 Exemple, page 73]{A04}). Observe that Legendre relation expresses the 2-period $ 2  \mathrm i \pi$ as the value of the polynomial $X_1Y_2-X_2Y_1$ evaluated in 1-periods.
	\end{example}
	
	 As we showed, the only $2g$-period $ (2 \mathrm i \pi )^g$ of a $g$-dimensional abelian variety is equal to $\pm \det(\Pi_A)$. In order to study these periods one should consider the whole De Rham and Hodge cohomology groups of $A$ and of its powers, and not just the  $\HH^1_{\dR}(A)$ and the $\HH_1(A(\CC),\mathbb Z)$.
 But, according to \cite[bottom of page 25]{D82} we have for $ r=0, \dots ,2g$ 
	\begin{equation}\label{kunneth}
		\begin{aligned}
			& 	\HH^r_{\dR}(A)= \bigwedge^r \HH^1_{\dR}(A) \qquad \mathrm{and} \qquad  \HH_r(A(\CC),\mathbb Z)= \bigwedge^r \HH_1(A(\CC),\mathbb Z)  \\
			&	\HH^r_{\dR}(A^n)= \bigoplus_{i_1+\dots +i_n=r} \HH^{i_1}_{\dR}(A) \otimes \dots \otimes  \HH^{i_n}_{\dR}(A) \\
			& \HH_r(A^n(\CC),\mathbb Z)= \bigoplus_{i_1+\dots +i_n=r}\HH_{i_1}(A(\CC),\mathbb Z) \otimes \dots \otimes \HH_{i_n}(A(\CC),\mathbb Z) ,\\
		\end{aligned}
	\end{equation}
which entails
 
	  \begin{lemma}\label{lemma:r-periods}
	 	 The $r$-periods of an abelian variety are the values of a degree $r$ polynomial evaluated in 1-periods.
	 \end{lemma}

\subsection{Periods of 1-motives}\label{periodsonemotives}

Let $M=[u: X \to G]$ be a $1$-motive defined over $K \subseteq \CC$, with $G$ a semi-abelian variety $0 \to Y(1) \to G\stackrel{\pi}{\to}A \to 0$. Recall that the choice of bases
 $x_1, \dots, x_{n}$ of the character group $X$ and $y^\vee_1, \dots , y^\vee_{s}$ of the character group $Y^\vee$  determines the points $P_1,\dots,P_n\in A(\oK)$, $Q_1,\dots,Q_s\in A^*(\oK)$, $R_1,\dots,R_n \in G(\oK)$ and an isomorphism $( y^\vee_1, \dots , y^\vee_{s}) :Y(1) \to \GG_m^s .$
In \cite[\S 2.3]{BPSS} to each point $Q_k,$ for $k=1, \dots,s,$ we have associated the invariant differential form on $G$
 \[\xi_{Q_k}^G = \tilde{\xi}_{Q_k} - \pi^* \xi_{Q_k},\]
 where $\tilde{\xi}_{Q_k}$ is the pull-back of the invariant differential form $\frac{dZ_k}{Z_k}$ of $\GG_m$ via the rational map $ G \cong A \times \GG_m^s \stackrel{pr_k}{\to} \GG_m,$ and $ \pi^* \xi_{Q_k}$ is the pull back via $\pi: G \to A$ of the differential form of the third kind $\xi_{Q_k}$ on $A$ associated to the point $Q_k$. Recall that the pull-back via the abelian exponential of $\xi_{Q_k}$ is the differential form $d \log (F_{q_k}(z)) ,$ where $F_{q_k}(z)$ is the function introduced in \eqref{Fq(z)}.
 In this paper we denote by $\xi_{Q_k}^G$ the opposite invariant differential form $-\tilde{\xi}_{Q_k} + \pi^* \xi_{Q_k}$ as it is more usual in the literature (see \cite[(1.11) and Proposition 2.3]{B19} \cite[\S 2]{Br08}).

 We complete the bases of $\mathrm H_1(A(\CC),\mathbb Z)$ and $\mathrm H^1_{\mathrm{dR}}(A)$ in ordered bases of the Hodge realization  $\T_\HH(M)$ and the De Rham realization $\T_\dR(M)$ of $M$ as:
\begin{equation}\label{basesforM}\begin{aligned}
&([O,R_1],\dots,[O,R_n],\gamma_1,\dots,\gamma_{2g},\delta_1,\dots,\delta_s)\\[3mm]
&(\xi_{Q_1}^G,\dots,\xi_{Q_s}^G,\pi^*\omega_1,\dots,\pi^*\omega_g,\pi^*\eta_1,\dots,\pi^*\eta_g,df_1,\dots,df_n)
\end{aligned}\end{equation}
where $df_i$ is an exact form such that $f_i(R_j)-f_i(O)=\delta_{i,j}$ for $i,j=1, \dots, n$, $\xi_{Q_k}^G$ is the invariant differential form on $G$  cited above,
 $[O,R_i]$ is a path in $G$ from the origin $O$ to the point $R_i$ which lifts 
to a path\footnote{As usual we will use small letters for abelian logarithms of points on $A(\CC),$ $A^*(\CC)$ or $G(\CC)$ which are written with capital letters, that is $\exp_{A}(p)=P$, ...} $[0,r_i]$ in $\mathrm{Lie}(G)\simeq\mathrm{Lie}(A)\times\mathrm{Lie}(\mathbb G_m)^s$ from the origin to a logarithm $r_i$ of $R_i=(P_i,e^{\ell_{i,1}},\dots,e^{\ell_{i,s}})$, $\gamma_1,\dots,\gamma_{2g}$ are loops that lift in $G$ the loops
$\lambda_1,\dots,\lambda_{2g}$ of $A$, and
  $\delta_k$ is a loop in $\GG_m^s $ such that $\int_{\delta_k} \xi_{Q_l}^G= 2 \mathrm i \pi \delta_{k,l}$ for $k,l=1, \dots,s$. The dual bases of $\T_\HH(M^*)$ and $\T_\dR(M^*)$ are
\begin{equation}\label{basesforM*}\begin{aligned}
&(\delta_1^*,\dots,\delta_n^*,-\gamma_{g+1},\dots,-\gamma_{2g},\gamma_1,\dots,\gamma_g,-[O,R_1^*],\dots, -[O,R_s^*])\\[3mm]
&(-df_1^*,\dots,-df_s^*,-\pi^*\eta_1,\dots,-\pi^*\eta_g,\pi^*\omega_1,\dots,\pi^*\omega_g,\xi_{P_1}^{G^*},\dots, \xi_{P_n}^{G^*})
\end{aligned}\end{equation}
with $[O,R_k^*]$ a path in $G^*$ from the origin $O$ to the point $R_k^*$ which lifts 
in a path $[0,r_k^*]$ in $\mathrm{Lie}(G^*)\simeq\mathrm{Lie}(A^*)\times\mathrm{Lie}(\mathbb G_m)^n$ from the origin to a logarithm  $r_k^*$ of $R_k^*=(Q_k,e^{\ell_{k,1}},\dots,e^{\ell_{k,n}}).$

The periods matrices of $M$ and $M^*$ with respect to these dual bases take the shape:
\begin{equation}\label{periodsmatrices}
\Pi_M=\left(\begin{matrix}
\Xi_R &\widetilde{\log}_A(P) &\mathrm {Id}_n\\
\Upsilon_Q &\Pi_A &0\\
2\mathrm i\pi\mathrm{Id}_s &0 &0
\end{matrix}\right)
\qquad
\Pi_{M^*}=\left(\begin{matrix}
0 &0 &2\mathrm i\pi\mathrm {Id}_n\\
0 &\Pi_{A^*} &\Upsilon_P\\
\mathrm{Id}_s &\widetilde{\log}_{A^*}(Q) &\Xi_{R^*}
\end{matrix}\right)
\end{equation}
with
$$\begin{matrix}\widetilde{\log}_A(P) = \left(\begin{matrix}\int_{z}^{p_i+z}\omega_j,\int_{z}^{p_i+z}\eta_j\end{matrix}\right)_{\stackrel{i=1,\dots,n}{\scriptscriptstyle j=1,\dots,g\hfill}} \enspace,\quad \widetilde{\log}_{A^*}(Q) = \left(\begin{matrix}\int_{z}^{q_k+z}\eta_j,-\int_{z}^{q_k+z}\omega_j\end{matrix}\right)_{\stackrel{k=1,\dots,s}{\scriptscriptstyle j=1,\dots,g\hfill}} \enspace,\\[6mm]
\Upsilon_Q = \left(\begin{matrix}\int_{\gamma_i}\xi^G_{Q_k}\end{matrix}\right)_{\stackrel{i=1,\dots,2g}{\scriptscriptstyle k=1,\dots,s\hfill}} \enspace,\quad \Upsilon_P = \left(\begin{matrix} -\int_{\gamma_{g+j}}\xi_{P_i}^{G^*}\\  \int_{\gamma_{j}}\xi_{P_i}^{G^*}\end{matrix}\right)_{\stackrel{j=1,\dots,g}{\scriptscriptstyle i=1,\dots,n\hfill}} \enspace,
\end{matrix}$$
and
$$\begin{matrix} \Xi_R = \left(\begin{matrix}\int_{O}^{R_i}\xi^G_{Q_k}\end{matrix}\right)_{\stackrel{i=1,\dots,n}{\scriptscriptstyle k=1,\dots,s\hfill}} \enspace,\quad  \Xi_{R^*} = \left(\begin{matrix}-\int_{O}^{R_k^*}\xi^{G^*}_{P_i}\end{matrix}\right)_{\stackrel{k=1,\dots,s}{\scriptscriptstyle i=1,\dots,n \hfill}}
\end{matrix}\enspace.$$

Observe that $\widetilde{\log}_A(P)$ (\textit{resp}. $\widetilde{\log}_{A^*}(Q)$) is the coordinates matrix of a generalized logarithm of $P$ (\textit{resp}. $Q$) in the given base of $\mathrm{H}^1_{\mathrm{dR}}(A)$ (\textit{resp}. $\mathrm{H}^1_{\mathrm{dR}}(A^*)$) (see \cite[Definition 2.2]{BPSS}). As observed in section Notation, if $K=\oK$ the extension $G$ underlying the 1-motive $M$
is isomorphic to $G_{Q_1}\times\dots\times G_{Q_s}$, with $G_{Q_k}$ the extension of $A$ by $\GG_m$ parametrized by the point $Q_k \in A^*(K)$ for $k=1, \dots s,$ and for $i=1, \dots, n$ the point $R_i \in G(K)$ corresponds to 
the points $R_{i,k}=\psi(x_i,y_k^\vee)$ living in $G_{Q_k}(K)$ above $P_i \in A(K)$. Choosing the bases of $\T_\HH(M)$ and $\T_\dR(M)$ compatible with this decomposition, the matrix $\Xi_R$ is then $\left(\int_O^{R_{i,k}}\xi^{G_{Q_k}}_{Q_k}\right)_{i,k}$. 
By \cite[Definition 2.4]{BPSS} the coordinates matrix of a generalized logarithm of $R$ is
\[ \tlog_{G}(R) = \big( \tlog_A(P),\Xi_R \big) =
\Big( \left(\begin{matrix}\int_{z}^{p_i+z}\omega_j,\int_{z}^{p_i+z}\eta_j\end{matrix}\right)_{\stackrel{i=1,\dots,n}{\scriptscriptstyle j=1,\dots,g\hfill}}, \left(\begin{matrix}\int_O^{R_{i,k}}\xi^{G_{Q_k}}_{Q_k}\end{matrix}\right)_{\stackrel{i=1,\dots,n}{\scriptscriptstyle k=1,\dots,s\hfill}} \Big).\]

In \cite[Lemma 3.2]{B02}, we have computed the periods matrix of a 1-motive without abelian part, i.e. $A=0.$ In the following example we compute the periods matrix of a 1-motive $M=[u: \ZZ \to G]$ with $n=s=1$ and $A$ an elliptic curve (see \cite[Proposition 2.3 and (2.7)]{B19} and \cite[\S 2]{Br08}). We will study this periods matrix more thoroughly in the last Section of this paper.

\begin{example}\label{example:XiwithoutA}
	The trivialization $\psi$ underlying a 1-motive $M=[u:X\to Y(1)]$ without abelian part is just a bilinear form $\psi : X \times Y^\vee \to \GG_m$ corresponding to the group morphism $u: X \to Y(1)$. In fact, if $\Psi:X \otimes Y^\vee \to \GG_m $ is the linear arrow associated to the bilinear form $\psi$, then  $u = (\Psi \otimes id_{Y}) \circ ( id_X \otimes \delta_Y): X  = X \otimes \ZZ \to X \otimes Y^\vee \otimes Y \to Y(1)  $ and reciprocally $\Psi = ev_Y \circ (u \otimes id_{Y^\vee}) : X \otimes Y^\vee \to Y(1) \otimes Y^\vee  \to  \ZZ \otimes \GG_m = \GG_m $. Denote by $x_1,\dots, x_n$ (\textit{resp.} $ y_1^\vee,\dots, y_s^\vee$) generators of the character group $X$ (\textit{resp.} $Y^\vee$). 	With respect to the chosen base in \eqref{basesforM} we have that
	$$ \Xi_R = \left(\begin{matrix} - \log \psi (x_i,y^\vee_k)\end{matrix}\right)_{\stackrel{i=1,\dots,n}{\scriptscriptstyle k=1,\dots,s\hfill}}$$
	and $u(x_i)= (\psi (x_i,y^\vee_1), \dots, \psi (x_i,y^\vee_s))$
	(see also \cite[Lemma 3.2]{B02} where $q_{ik} :=\psi (x_i,y^\vee_k) $).
	The Cartier dual of $M$ is the 1-motive  $M^*=[u^*:Y^\vee \to X^\vee(1)],$ 
	where the group morphism $u^*$ corresponds to the bilinear form $\psi \circ \sym_{Y^\vee \times X} : Y^\vee \times X \to \GG_m$ (see \cite[(10.2.12.2)]{D74}).  In particular, $ u^*(y^\vee_k)= (\psi (x_1, y^\vee_k), \dots, \psi (x_n,y^\vee_k)).$
	With respect to the chosen base in \eqref{basesforM*} we get
	$$ \Xi_{R^*}
	= \left(\begin{matrix} \log \psi (x_i,y^\vee_k)\end{matrix}\right)_{\stackrel{k=1,\dots,s}{\scriptscriptstyle i=1,\dots,n\hfill}} 
	=- \left(\begin{matrix}  - \log  \psi (x_i,y^\vee_k)\end{matrix}\right)_{\stackrel{k=1,\dots,s}{\scriptscriptstyle i=1,\dots,n\hfill}} = -  \Xi_R^t. $$
	Hence we have the equality $ \Xi_{R^*} +\Xi_R^t =0$ that will be generalized in Lemma \ref{computperiodes}. The periods matrices of $M$ and $M^*$ are
	\begin{equation*}
		\Pi_M=\left(\begin{matrix}
			\Xi_R &\mathrm {Id}_n\\
			2\mathrm i\pi\mathrm{Id}_s  &0
		\end{matrix}\right)
		\qquad \qquad 
		\Pi_{M^*}=\left(\begin{matrix}
			0  &2\mathrm i\pi\mathrm {Id}_n\\
			\mathrm{Id}_s &\Xi_{R^*}
		\end{matrix}\right)
	\end{equation*}
\end{example}

\begin{example}\label{example:matriceMsemi-elliptique}
	Let $M=[u: \ZZ \to G]$ be a 1-motive defined over $K \subseteq \CC$, where $G$ is an extension of an elliptic curve $\cE$ by $\GG_m$ parametrized by a point  $Q \in \cE^*(\oK)$, and where the group morphism $u$
	is defined by a point $R=(P, e^l) \in G(\oK)$ living above a point
	$P\in \cE(\oK)$.
	As it is more usual in the literature on elliptic curves, in Example \ref{example:GalMotE}, in Subsection \ref{subsubsection4} and here we change slightly the notations: 
	$\omega_1, \omega_2$ denote the periods of the Weierstrass $\wp$-function and $\eta_1,\eta_2$ denote the quasi-periods of the Weierstrass $\zeta$-function.
	Let $ \sigma(z)$ be the Weierstrass $\sigma$-function relative to the lattice $\Lambda \cong \HH_1(\cE(\CC),\ZZ)$.
	In the elliptic case, the function $F_{q}(z)$ associated to the point $Q$ in \eqref{Fq(z)} in order to construct a differential form of the third kind $\xi_{Q}$ on $\cE$, is replaced by
	Serre's function 
	$	f_{q}(z)= \frac{\sigma(z+q)}{\sigma(z) \sigma(q)} e^{-\zeta(q) z } .$
	The exponentials of the elliptic integrals of the third kind
	\begin{equation}\label{eq:periods-fq}
		\int_{\lambda_i} \xi_{Q} = \eta_i q - \omega_i \zeta(q) \qquad \qquad i=1,2.
	\end{equation}
	are the so called ``quasi-quasi periods" of the function $f_q(z):$ in fact $
	f_{q}(z+ \omega_i)= f_{q}(z) e^{\eta_i q - \omega_i \zeta(q)}. $
	If $p_1$ and $p_2$ are two points of $\Lie \cE_\CC$ such that $p=p_2 -p_1$, we have that
	\begin{align} 
		\nonumber		\int_O^{R}  \xi_{Q}^G &=  \int_O^{R}  (-\tilde{\xi}_{Q} + \pi^* \xi_{Q}) = -\int_0^{\ell} dt +\int_{p_1}^{p_2} \frac{f_{q}'(z)}{f_{q}(z)} dz = -\ell +	\int_{p_1}^{p_2} d \log f_{q}(z) \\
		\nonumber		& = - \ell + \log \frac{f_{q}(p+p_1)}{f_{q}(p) f_q(p_1)} + \log f_{q}(p)
	\end{align}
	Since by \cite[(1.9)]{B19} and \cite[(3.13.4)]{Breen}  $\frac{f_x(y+z)}{f_x(y)f_x(z)}=\frac{\sigma(x+y+z) \sigma(x) \sigma(y) \sigma(z)}{\sigma(x+y) \sigma(x+z) \sigma(y+z)} \in K(\cE)$, it exists a rational function $g$ on $\cE$ such that $\log \frac{g(p_2)}{g(p_1)} = - \log \frac{f_{q}(p_2)}{f_{q}(p) f_{q}(p_1)}$. Adding to the class of $\pi^* \xi_{Q}$ the differential $d \log g$ we get
	\[  \int_O^{R}  (-\tilde{\xi}_{Q} + \pi^* \xi_{Q} +d \log g)=-\ell +  \log f_{q}(p).\]
	With respect to the dual bases introduced at the beginning of this section, the periods matrices of $M$ and $M^*$ are
	\begin{equation*}
		\Pi_M=	\left( {\begin{array}{cccc}
				\log f_q(p) - \ell &p & \zeta(p)&1 \\
				\eta_1 q - \omega_1 \zeta(q) & \omega_1 &  \eta_1 & 0 \\
				\eta_2 q - \omega_2 \zeta(q) & \omega_2 & \eta_2 &0  \\
				2 \mathrm i \pi &0 & 0 & 0\\
		\end{array} } \right), \quad 
		\Pi_{M^*}=	\left( {\begin{array}{cccc}
				0 &0 & 0 & 2 \mathrm i \pi \\
				0 & \eta_2 &  - \omega_2 & -\eta_2 p + \omega_2 \zeta(p) \\
				0 & - \eta_1 & \omega_1 & 	\eta_1 p - \omega_1 \zeta(p) \\
				1 & \zeta(q) & -q & -\log f_p(q) + \ell \\
		\end{array} } \right).
	\end{equation*}
\end{example}

\begin{notation}\label{remark:gamma}
	Modifying the loops $\gamma_1, \dots , \gamma_{2g}$, which lift $\lambda_1, \dots , \lambda_{2g}$ in $G$, by elements of $\HH_1(\GG_m^s , \ZZ),$ we can assume that $\int_{\gamma_i} \tilde{\xi}_{Q_k}=0.$ From now on we make this choice.
\end{notation}

The equalities \eqref{eq:periods-fq} read  $\Upsilon_Q = - \Pi_A\big(\widetilde{\log}_{A^*}(Q)\big)^{t}$. Remark that in the above Examples we have also the equalities $\Upsilon_P = - \Pi_{A^*}(\widetilde{\log}_{A}(P))^t$ and
$\Xi_{R}^t +\Xi_{R^*} = - \widetilde{\log}_{A^*}(Q)\big(\widetilde{\log}_A(P)\big)^t$.  Now we prove that these equalities are true in general.

\begin{lemma}\label{computperiodes}
With the above notations we have
\begin{itemize}
	\item  $\Upsilon_Q = - \Pi_A\big(\widetilde{\log}_{A^*}(Q)\big)^{t}$,
	\item  $\Upsilon_P = - \Pi_{A^*}(\widetilde{\log}_{A}(P))^t$, and
	\item $\Xi_{R}^t +\Xi_{R^*} = - \widetilde{\log}_{A^*}(Q)\big(\widetilde{\log}_A(P)\big)^t$.
\end{itemize}

\end{lemma}
\begin{proof}
The coefficient of index $(i,k  )$ of the matrix $\Upsilon_Q$ is
$$\begin{aligned}\int_{\gamma_i} \xi^G_{Q_k} &= \int_z^{\lambda_i+z}d\log(F_{q_k})\\[2mm]
&= \log\big(F_{q_k}(\lambda_i+z)/F_{q_k}(z)\big)\\ 
&= \pi H(q_k,\lambda_i) - \sum_{j=1}^g h_j(q_k)\int_z^{\lambda_i+z}\omega_{j}\\
&= - \sum_{j=1}^g\left(-\pi\frac{d}{dz_j}H(z,\lambda_i)|_{z=q_k}\int_z^{q_k+z}\omega_j + h_j(q_k)\int_z^{\lambda_i+z}\omega_{j}\right)\\
&= - \sum_{j=1}^g \left(-\int_z^{\lambda_i+z}\eta_{j}\int_{z}^{q_k+z}\omega_j + \int_{z}^{q_k+z}\eta_j\int_z^{\lambda_i+z}\omega_{j}\right)\end{aligned}$$
since $H(z,\lambda_i)$ is a linear form in $z$ and $\int_z^{\lambda_i+z}\eta_j = h_j(\lambda_i+z)-h_j(z) = \pi \frac{d}{dz_j}H(z,\lambda_i)$ does not depend on $z$. While the matrix on the right hand side is
\begin{equation*}
 - \Pi_A\left(\begin{matrix}\int_{z}^{q_k+z}\eta_j\\
-\int_{z}^{q_k+z}\omega_j\end{matrix}\right)_{\stackrel{j=1,\dots,g}{\scriptscriptstyle k=1,\dots,s\hfill}} = - \left(\begin{matrix} \sum_{j=1}^g\left(\int_z^{\lambda_i+z}\omega_j\int_z^{q_k+z}\eta_j - \int_z^{\lambda_i+z}\eta_j\int_z^{q_k+z}\omega_j\right)\end{matrix}\right)_{\stackrel{i=1,\dots,2g}{\scriptscriptstyle k=1,\dots,s\hfill}}
\end{equation*}
and the desired equality follows. A similar computation yields the second equality.

As for the third identity, we observe that $\int_{O}^{R_i}  \tilde{\xi}_{Q_k} - \int_{O}^{R^*_k} \tilde{\xi}_{P_i}$ is zero (see Example \ref{example:XiwithoutA}).
Therefore the coefficient of index $(k,i)$ of the matrix $\Xi_{R}^t +\Xi_{R^*}$ is   
$$\begin{aligned}
\int_{O}^{P_i}\xi_{Q_k} - \int_{O}^{Q_k}\xi_{P_i} &= \int_{z}^{p_i+z}d\log(F_{q_k}) - \int_{z}^{q_k+z}d\log(F_{p_i})\\
&= - \log\big(F_{p_i}(q_k)/F_{q_k}(p_i)\big)\\
&= - \sum_{j=1}^g\left(h_j(q_k)\int_z^{p_i+z}\omega_j - h_j(p_i)\int_z^{q_k+z}\omega_j\right)\\
&= - \sum_{j=1}^g\left(\int_z^{q_k+z}\eta_j\int_z^{p_i+z}\omega_j - \int_z^{p_i+z}\eta_j\int_z^{q_k+z}\omega_j\right).
\end{aligned}$$
Hence $\Xi_{R}^t +\Xi_{R^*}= - \widetilde{\log}_{A^*}(Q)(\widetilde{\log}_{A}(P))^t
$.
\end{proof}

\begin{remark}\label{rk:ChoixChemins}
	The matrix $\Pi_M$ of \eqref{periodsmatrices} depends among other things on the choice of the logarithms of $P,Q$ and $R.$ Let $\alpha_1, \dots , \alpha_n$ (\textit{resp.} $\alpha^*_1, \dots, \alpha^*_s$) be the coordinates in the base of $\Lambda^n$ (\textit{resp.} $\Lambda^{*s}$) of the difference of two logarithms of $P \in A$ (\textit{resp.} $Q \in A^*$). Then, the difference of the corresponding two 
	\[
	\bullet \; \widetilde{\log}_A(P) \; \mathrm{is}  \left(\begin{matrix} \alpha_1 \Pi_A \\ \vdots\\ \alpha_n \Pi_A \end{matrix}\right) \qquad
	\bullet \;  \widetilde{\log}_{A^*}(Q) \; \mathrm{is} \left(\begin{matrix} \alpha_1^* \Pi_{A^*} \\ \vdots\\ \alpha_s^* \Pi_{A^*} \end{matrix}\right) \qquad
	\bullet \;  \Upsilon_Q  \;\mathrm{is} - 2 \mathrm{i} \pi(\alpha^{*t}_1, \dots, \alpha^{*t}_s) ,
	\]
	where for the last assertion we have used the second one and Lemma \ref{computperiodes}.
	Let $(\widetilde{\alpha}_1, \dots, \widetilde{\alpha}_n )$  be the coordinates in the base of $\HH_1(G(\CC) , \ZZ)^n$ of the difference of two logarithms of $R \in G^n$. 
	For $i=1, \dots ,n$, we can write $\widetilde{\alpha}_i = (\alpha_i, \sigma_i)$, where $ \alpha_i$ is the image of $\widetilde{\alpha}_i $ via the surjection $ \HH_1(G(\CC) , \ZZ) \to  \HH_1(A(\CC), \ZZ) $ and $\sigma_i \in \HH_1(\GG_m^s(\CC) , \ZZ).$
	Hence the difference of two 
	\begin{itemize}
		\item $\Xi_R$ is $\left(\begin{matrix} \alpha_1  \\ \vdots\\ \alpha_n \end{matrix}\right) \Upsilon_Q + 2 \mathrm i \pi  \left(\begin{matrix} \sigma_1  \\ \vdots\\ \sigma_n \end{matrix}\right)$ 
	\end{itemize}
	since $\left(\begin{matrix}\int_{\lambda_i}\xi_{Q_k}\end{matrix}\right)_{i} = \Upsilon_Q$ by Notation \ref{remark:gamma}.
\end{remark}

We also compute
$${2\mathrm i\pi}\Pi_{M}^{-1}=\left(\begin{matrix}
0 &0 &\mathrm{Id}_s\\
0 &{2\mathrm i\pi}\Pi_{A}^{-1} &-\Pi_A^{-1}\Upsilon_Q\\
{2\mathrm i\pi}\mathrm {Id}_n &{-2\mathrm i\pi}\widetilde{\log}_{A}(P)\Pi_A^{-1} &\widetilde{\log}_A(P)\Pi_A^{-1}\Upsilon_Q-\Xi_{R}
\end{matrix}\right).$$

\begin{corollary}\label{cor:lien}
With the choices of bases in \eqref{basesforM} and \eqref{basesforM*} one has $\Pi_{M^*} = (2\mathrm i\pi\Pi_M^{-1})^t$. In particular $\det(\Pi_{M}) = \pm(2\mathrm i\pi)^{g+s}$ and $\det(\Pi_{M^*}) = \pm(2\mathrm i\pi)^{n+g}$.
\end{corollary}
\begin{proof}
According to the Lemma \ref{computperiodes}, we have 
\begin{itemize}
	\item  $- \Pi_A^{-1}\Upsilon_Q=   (\widetilde{\log}_{A^*}(Q))^t,$
	\item $- 2\mathrm i\pi\widetilde{\log}_A(P)\Pi_A^{-1} =- \widetilde{\log}_A(P)\Pi_{A^*}^t =  \Upsilon_P^t,$ 
	\item $\widetilde{\log}_A(P)\Pi_A^{-1}\Upsilon_Q - \Xi_R = -\widetilde{\log}_A(P)(\widetilde{\log}_{A^*}(Q))^t - \Xi_R
	= \Xi_{R^*}^t$,
\end{itemize}
 which show the first equality. For the last equalities observe that
  $ \det(\Pi_M)= (-1)^{ns}(2\mathrm{i}\pi)^s\det(\Pi_A),$ 
 $ \det(\Pi_{M^*})=(-1)^{ns}(2\mathrm{i}\pi)^n\det(\Pi_{A^*})$
  and $\det(\Pi_{A^*})=\det(\Pi_{A})=  \pm(2\mathrm{i}\pi)^g$.
\end{proof}

Confronting the two matrices in \eqref{periodsmatrices}, as a direct consequence of Lemma \ref{computperiodes} we have

\begin{corollary}\label{cor:egalitedecorps}
The fields generated over $K$ by the periods of $M$ and $M^*$ respectively, coincide. 
\end{corollary}

\begin{remark}\label{rem:QPi_MK}
The \emph{periods matrices} of the $1$-motive $M$ defined over a subfield $K\subset\CC$ are the elements of the double class $\mathrm{Gl}_{n+2g+s}(\mathbb Z)\Pi_M\mathrm{Gl}_{n+2g+s}(K)$.
\end{remark}

\subsection{Representation of the motivic Galois group}\label{repmotGalois}
Assume the 1-motive $M=[u:\mathbb Z^n\to G]$ to be defined over $ K= \oK \subseteq \CC$.
 Denote by $<M>^\otimes$ the tannakian sub-category generated by $M$ in a suitable tannakian category of motives.
 As in Section \ref{periodsonabelianvarieties}, denote by $\underline{\mathrm{Isom}}^\otimes( {\T_\dR}_{|<M>^\otimes }, {\T_\HH}_{|<M>^\otimes } ) $ the functor over the $K$-schemes, which associates to any $f: T \to \mathrm{Spec}(K)$ the set of isomorphisms of $\otimes$-functors  $f^*{\T_{\dR}}_{|<M>^\otimes } \to f^*{ \T_\HH}_{|<M>^\otimes } ,$ where  $ {\T_\HH}_{|<M>^\otimes }$ and ${\T_\dR}_{|<M>^\otimes }$ are respectively the fibre functor ``Hodge realization" and ``De Rham realization" of $M$.
 It is representable by an affine $K$-scheme and the periods matrix $\Pi_M$ in \eqref{periodsmatrices} is a $\CC$-rational point of this $K$-scheme: in particular $\underline{\mathrm{Isom}}^\otimes(  {\T_\dR}_{|<M>^\otimes }, {\T_\HH}_{|<M>^\otimes }) (\CC)$ is not empty.
Moreover  $\MT(M) \otimes K ={\underline {\Aut}}^{\otimes}_\QQ({\T_\HH}_{|<M>^\otimes }) \otimes K$ acts on $\underline{\mathrm{Isom}}^\otimes( { \T_\dR}_{|<M>^\otimes }, {\T_\HH}_{|<M>^\otimes }) $: for any $f: T \to \mathrm{Spec}(K)$, $F \in \underline{\mathrm{Isom}}^\otimes( { \T_\dR}_{|<M>^\otimes }, {\T_\HH}_{|<M>^\otimes } )(T)$ and $G \in ( {\underline {\Aut}}^{\otimes}_\QQ({\T_\HH}_{|<M>^\otimes }) \otimes K)(T)$, the action of $G$ on $F$ is given by the composite
\[f^*{\T_{\dR}}_{|<M>^\otimes } \stackrel{F}{\longrightarrow}   f^* {\T_\HH}_{|<M>^\otimes }  \stackrel{G}{\longrightarrow}  f^* {\T_\HH}_{|<M>^\otimes }. \]
Hence we can conclude that the affine $K$-scheme $\underline{\mathrm{Isom}}^\otimes(  {\T_\dR}_{|<M>^\otimes }, {\T_\HH}_{|<M>^\otimes } ) $ is in fact a $\MT(M) \otimes K$-torsor, called the \emph{torsor of periods} of the 1-motive $M$ (see \cite[Proposition 1.6]{D82}).


 Denote by $\MT(M) \Pi_M$
the orbit of the point $\Pi_M$ under the action of $\MT(M),$ and by
 $\overline{\Pi_M}^Z $
 the $K$-Zariski closure of the point $\Pi_M$ in $\underline{\mathrm{Isom}}^\otimes( { \T_\dR}_{|<M>^\otimes }, {\T_\HH}_{|<M>^\otimes } ) .$ If the 1-motive $M$ is defined over $\oQQ$, Grothendieck periods conjecture states that (see Andr\'e's letter in \cite[Appendix]{B19})
 \begin{enumerate}
 	\item the torsor of periods $\underline{\mathrm{Isom}}^\otimes( { \T_\dR}_{|<M>^\otimes }, {\T_\HH }_{|<M>^\otimes }) $ is connected, and 
 	\item the dimension of
 	$ \overline{\Pi_M}^Z $ is equal to the dimension of $ \underline{\mathrm{Isom}}^\otimes(  {\T_\dR}_{|<M>^\otimes }, {\T_\HH}_{|<M>^\otimes } ).$
 \end{enumerate}
 In particular, this conjecture implies that $ \dim \overline{\Pi_M}^Z = \dim \MT(M).$

 The $\oQQ$-points of $\MT(M)$ act on the complex point $\Pi_M$ of $\underline{\mathrm{Isom}}^\otimes(  {\T_\dR}_{|<M>^\otimes },{ \T_\HH}_{|<M>^\otimes })$. These points of $\MT(M)$  are represented by the elements of $\mathrm{Gl}_{n+2g+s}(\oQQ)$ (acting by multiplication on the left of the given periods matrix) which leave invariant the ideal of algebraic relations with coefficients in $\overline{\mathbb Q}$, coming from motivic cycles,
 	between the entries of $\Pi_M$. See Section \ref{subsubsection4} for a computation of algebraic relations coming from geometry  between the periods of $M$ and for an explicit description in terms of matrices of its Mumford-Tate group in the semi-elliptic case with $n=s=1.$

  Because of the shape of the periods matrix \eqref{periodsmatrices}, a matrix preserving the ideal of algebraic relations between the entries of $\Pi_M$ with coefficients in $\overline{\mathbb Q}$, must have the shape
\begin{equation}\label{matmotGalois}
	\rho(a,u,u^*,\sigma) := \left(\begin{matrix}
		\mathrm{Id}_n &u &\sigma\\
		0 &a &u^*\\
		0 &0 & \det (a)^{1/g} \mathrm{Id}_s
	\end{matrix}\right)
\end{equation}
with $u\in\mathrm{Mat}_{n,2g}(\oQQ)$, $u^*\in\mathrm{Mat}_{2g,s}(\oQQ)$, $\sigma\in\mathrm{Mat}_{n,s}(\oQQ)$ and $a\in\mathrm{Gl}_{2g}(\oQQ)$. But not all matrices of this shape fix the ideal of algebraic relations among the periods of $M$ coming from algebraic cycles on $M$ and its powers.


In terms of matrices the group law of $\MT(M)$ has the following description:
\begin{equation}\label{eq:grouplaw}
\rho(a,u,u^*,\sigma)\rho(b,v,v^*,\tau) = \rho(ab,v+ub, av^* +u^*  \det (b)^{1/g}, \tau +uv^* + \sigma  \det (b)^{1/g}).
\end{equation}
The group $\MT(M)$ is a subextension of the Heisenberg group $\mathcal{H}_M$  introduced in \cite[\S 1.3]{B01}, whose factors system is $((u,u^*),(v,v^*)) \mapsto uv^*.$
The action of $\MT(M)(\CC)$ on the $\CC$-rational point $\Pi_M$ of $\underline{\mathrm{Isom}}^\otimes( \T_\dR, \T_\HH ) $ reads
\begin{equation}\label{actionGalmot}
	\rho(a,u,u^*,\sigma)\Pi_M = \left(\begin{matrix}\Xi_R+u\Upsilon_Q+2\mathrm i\pi\sigma &\widetilde{\log}_A(P)+u\Pi_A &\mathrm{Id}_n\\ a\Upsilon_Q+2\mathrm i\pi u^* &a\Pi_A &0\\ 2\mathrm i\pi \det(a)^{1/g} \mathrm{Id}_s &0 &0\end{matrix}\right)
	.\end{equation}

 The map $\MT(M)\to\MT(A)$, described in terms of matrices as $\rho(a,u,u^*,\sigma)\mapsto a$, is a group morphism the kernel of which is the unipotent radical $\UR(\MT(M)) $ of the Mumford-Tate group of $M$, since $\MT(A)$ is reductive \cite[\S 3.1, bullet (3.6)]{BPSS}.
 In particular $\UR(\MT(M))$ consists of unipotent elements that are represented by matrices as in \eqref{matmotGalois} with $a=\mathrm{Id}_{2g}$. We finish this Subsection with an example of Mumford-Tate group that will be re-examined in the last section.

	\begin{example}\label{example:GalMotE}
	The Mumford-Tate group of an elliptic curve $\cE$ defined over $K=\oK$ is
	\begin{itemize}
		\item $\mathrm{Gl}_2(\mathbb Q)$ if $\cE$ has no complex multiplication;
		\item $\mathbb G_m^2$ if $\cE$ has complex multiplication. In this case $\oK = \overline{\QQ}$ and $\tau:=\omega_2/\omega_1$ is a quadratic number. The periods matrix $\Big(\begin{matrix}\omega_1&\eta_1\\\omega_2&\eta_2\end{matrix}\Big)$ is then a zero of the ideal generated by the two linear forms $\omega_2-\tau\omega_1$ and $|\tau|^2\eta_1-\tau\eta_2-\kappa\omega_1$ for some $\kappa\in\overline{\mathbb Q}$, see \cite[Chap. III, \S3.2, Lemma 3.1, page 36] {M75}. The Mumford-Tate group of $\cE$ is therefore represented by the group of matrices $ \left(\begin{matrix}b_1 &b_2\\ - b_2|\tau|^2 &b_1+b_2(\tau+\overline{\tau})\end{matrix}\right)$ which are invertible and leave invariant the above ideal. After extending the scalars to $\QQ(\tau),$ we may identify the Mumford-Tate group of $\cE$ with the diagonal matrices. Indeed the above representation is then isomorphic to the following  conjugate
		$$ \frac{1}{\tau-\overline{\tau}}
		\left(\begin{matrix} 1 & -1\\ \tau & -\overline{\tau}\end{matrix}\right)
		\left(\begin{matrix} a_1 & 0\\ 0 & a_2 \end{matrix}\right) \left(\begin{matrix} -\overline{\tau} & 1\\ -\tau & 1\end{matrix}\right)
		= \frac{1}{\tau-\overline{\tau}}\left(\begin{matrix}a_2\tau-a_1\overline{\tau} &a_1-a_2\\ (a_2-a_1)|\tau|^2 &a_1\tau-a_2\overline{\tau}\end{matrix}\right).$$
	\end{itemize}
\end{example}

\subsection{Representation of the unipotent radical}\label{repsunipotentrad}

As in the previous Subsection assume the 1-motive $M$ defined over $ K= \oK \subseteq \CC$. By \eqref{eq:ses},
the Lie algebra of the unipotent radical of the motivic Galois group $\Galmot(M)$ of $M$ is the extension 
\[
	0 \longrightarrow Z(1) \longrightarrow \Lie \UR (M) \longrightarrow B \longrightarrow 0 
\]
whose factors system is the restriction to $B \times B$ of the map \eqref{factorsystem}
(recall that $B$ and $Z(1)$ have been introduced in Section \ref{subsection:BZ}). The image of the above short exact sequence via the fibre functor ``Hodge realization" is the short exact sequence of mixed Hodge structures over $\QQ$
\begin{equation}\label{T_Hexactsequence}
	0 \longrightarrow \HH_1(Z(1)(\CC),\QQ) \longrightarrow \UR (\MT(M))(\QQ) \longrightarrow \HH_1(B(\CC),\QQ) \longrightarrow 0 .
\end{equation}

We fix symplectic bases of $\Lambda $ and $\Lambda^*$ (as in Subsection \ref{periodsonabelianvarieties}) which are dual bases for the form $\mathrm{Im} \HH$.
Since $\HH_1(B(\CC),\QQ) $ is contained in $\HH_1(A(\CC),\QQ)^n \times \HH_1(A^*(\CC),\QQ)^s$, we write an element $\log_B(b)$ of $\HH_1(B(\CC),\QQ) $ in two different ways:
\begin{itemize}
	\item as vector $(u,u^*):= (u_1, \dots, u_n,u^*_1, \dots , u^*_s)$ of  $\HH_1(A(\CC),\QQ)^n \times \HH_1(A^*(\CC),\QQ)^s$, with $u_i \in \HH_1(A(\CC),\QQ)  $ and $u_k^* \in \HH_1(A^*(\CC),\QQ)$ for $ i=1, \dots,n $ and $ k=1, \dots,s$, 
	\item as matrix $(u,u^*)$ of symplectic coordinates, with
		\[u = \left(\begin{matrix} \underline{u}_1 \\ \vdots\\ \underline{u}_n \end{matrix}\right)  \in  \mathrm{Mat}_{n,2g}(\RR), \qquad
\big(\mathrm{resp.} \;	u^{*}= (\underline{u}_{ 1}^*, \dots,\underline{u}_{ s}^*)  \in  \mathrm{Mat}_{2g,s}(\RR) \big).\]
	where $\underline{u}_i \in \mathrm{Mat}_{1,2g}(\RR)$ (resp. $\underline{u}_k^* \in \mathrm{Mat}_{2g,1}(\RR)$) are the coordinates with respect to the chosen symplectic base of $\Lambda$ (resp. of $\Lambda^*$) of the point $u_i \in \HH_1(A(\CC),\QQ)  $ for $ i=1, \dots,n $ (resp. $u_k^* \in \HH_1(A^*(\CC),\QQ)$ for $ k=1, \dots,s$).
\end{itemize}
 Recalling the Hodge realization \eqref{eq:hodge} of the motivic Weil Pairing, by \eqref{factorsystem} the factors system defining the extension \eqref{T_Hexactsequence} is the map
\begin{align}\label{eq:LieRealizationFactorSystem}
\HH_1(B(\CC),\QQ) \times \HH_1(B(\CC),\QQ) & \longrightarrow \HH_1(Z(1)(\CC),\QQ) ,\\
	\nonumber \big( (u,u^*), (v,v^*) \big)& \longmapsto \big( \mathcal{W}^\HH_A (u_{i},v^*_{k})  \big)_{ i=1, \dots,n \atop k=1, \dots,s}. 
\end{align}

  \begin{definition}\label{defG}
  	Let $\mathcal G$ be the subgroup  of $\mathrm{Gl}_{n+2g+s}(\mathbb Q)$  consisting of unipotent upper triangular matrices $\rho(\mathrm{Id}_{2g},u,u^*,\sigma) $ as in \eqref{matmotGalois} such that 
  	\begin{itemize}
  		\item $(u,u^*)$ is the coordinates matrix (with respect to the chosen symplectic bases of $\Lambda $ and $\Lambda^*$) of a logarithm $\log_B(b) \in \HH_1(B(\CC),\QQ) $ of a point $b \in B(\CC)$,
  		\item $\sigma$ is the coordinates matrix of a logarithm $\log_{Z(1)}(T)\in \HH_1(Z(1)(\CC),\QQ) $ of a point $T \in Z(1)(\CC).$ 
  	\end{itemize}
  \end{definition}

 Let $\mathcal T $  be the subgroup of $\mathcal G$ of unipotent upper triangular matrices $\rho(\mathrm{Id}_{2g},0,0,\sigma) $ and denote by $\mathcal A $ the quotient  of $\mathcal G$ by $\mathcal T $.
Clearly $\mathcal G$ is an extension of $\mathcal A $ by $\mathcal T$ and  the factors system of this extension is given by the group law \eqref{eq:grouplaw}.

 \begin{proposition}\label{propGRUMTM}
 	The unipotent radical $\UR (\MT(M))(\QQ)$ of the Mumford-Tate group of $M$ is isomorphic as extension to the subgroup $\mathcal G$ of $\mathrm{Gl}_{n+2g+s}(\mathbb Q)$ introduced above.  
 \end{proposition} 
 
 In order to prove this Proposition we need the following Lemma:

 \begin{lemma}\label{lemma:factorsystem}
 	The factors system \eqref{eq:LieRealizationFactorSystem} of the unipotent radical $\UR (\MT(M))$  coincides with the factors system of the extension $\mathcal G$ given by the group law \eqref{eq:grouplaw}.
 \end{lemma}
 
 \begin{proof} Consider two elements  $\log_B(b_1) =(u,u^*):= (u_1, \dots, u_n,u^*_1, \dots , u^*_s)$ and $\log_B(b_2) =(v,v^*):= (v_1, \dots, v_n,v^*_1, \dots , v^*_s)$ of  $\HH_1(B(\CC),\QQ)$.
Since the bases of $\Lambda$ and $\Lambda^*$ are dual for the duality product \eqref{dualityproduct}, by \eqref{eq:hodge} we have that
\[ \big( \mathcal{W}^\HH_A (u_{i},v^*_{k})  \big)_{ i,k} =
\Big( \mathrm{Im} \big(v^*_{ k}(u_{ i})\big)  \Big)_{i,k} 
=  \big(  \, \underline{u}_i \;  \underline{v}^*_k \big)_{i,k} =	u v^*.
\] 	
Hence the image of the point $( \log_B(b_1), \log_B(b_2))$ via the factors system \eqref{eq:LieRealizationFactorSystem} is $u v^*.$	
 	 \end{proof}

\begin{proof}[Proof of Proposition \ref{propGRUMTM}] Let $\mathcal{H}_M = \HH_1(A(\CC),\QQ)^n \times \HH_1(A^*(\CC),\QQ)^s \times \HH_1(\GG_m(\CC),\QQ)^{ns}$ be the generalized Heisenberg group associated to $M$ in \cite[1.3]{B01}, which has group law \eqref{eq:grouplaw}. Identifying \\ $\rho(\mathrm{Id}_{2g},u,u^*,\sigma)$ with the element $ (u,u^*,\sigma) $ of $ \HH_1(B(\CC),\QQ) \times \HH_1(Z(1)(\CC),\QQ)$, we view $\mathcal G$ as a subgroup of $\mathcal{H}_M$ which is an extension of $\HH_1(B(\CC),\QQ)$ by $\HH_1(Z(1)(\CC),\QQ)$. By \cite[Structural Lemma 1.4 (2)]{B01} and by \eqref{T_Hexactsequence},  $\UR(\MT(M))$ is a subgroups of $\mathcal{H}_M$ which is an extension of $\HH_1(B(\CC),\QQ)$ by $\HH_1(Z(1)(\CC),\QQ)$. According to  Lemma \ref{lemma:factorsystem} these two groups have the same factors system, and so they are isomorphic.
\end{proof}

\begin{remark} The subgroup $\mathcal{T}$ of $\mathcal{G}$ is isomorphic to $\W_{-2}(\UR(\MT(M)))$ and the quotient  $\mathcal{A}$ of $\mathcal{G}$ by $\mathcal{T}$ is isomorphic to $\UR(\MT(M \oplus M^* / \W_{-2}(M \oplus M^*)))$.
\end{remark}

\begin{lemma}\label{lemma[]} 
	Let $(u,u^{*})$ and $(v,v^{*})$ be the coordinates matrices (with respect to the chosen symplectic bases of $\Lambda$ and $\Lambda^*$) of two logarithms $ \log_B(b_1),\log_B(b_2) \in \HH_1(B(\CC),\QQ)$ respectively.
	Then applying $e^{2 \mathrm i \pi \cdot}$  to the coefficients of the matrix
	\[u v^* - v u^*\]
	 we obtain the Lie realization of the Lie bracket \eqref{eq:LieRealization[,]} of the point $ (\log_B(b_1), \log_B(b_2))$.	
\end{lemma}

\begin{proof} Consider two elements  $\log_B(b_1) =(u,u^*)= (u_1, \dots, u_n,u^*_1, \dots , u^*_s)$ and $\log_B(b_2) =(v,v^*)= (v_1, \dots, v_n,v^*_1, \dots , v^*_s)$ of  $\HH_1(B(\CC),\QQ)$.
Proceeding as in Lemma \ref{lemma:factorsystem}, by the expression of the Lie realization of the Lie bracket \eqref{eq:LieRealization[,]}, we have on one hand
\[(	\big[(u,u^*),(v,v^*)\big] )_{i,k}=
e^{	2{\rm i}\pi{\rm Im} \big( v^*_{ k} (u_{ i}) - u^*_{k} (v_{i}) \big) } =  e^{ 2 \mathrm i \pi \big( \underline{u}_{i} \, 
	\underline{v}_{ k}^*  -  \underline{v}_i \, \underline{u}_k^* \big)} .
	\] 
	 On the other hand
	$	u v^* - v u^*=  \big( \underline{u}_{i} \, 
	\underline{v}_{ k}^*  -  \underline{v}_i \, \underline{u}_k^* \big)_{i,k} .$
\end{proof}

\begin{remark}
With notations of Subsection \ref{subsectionZ'(1)}, let $j: B \to A^n\times {A^*}^s$ be the inclusion and let $\varpi:A^n\times{A^*}^s\twoheadrightarrow C:=(A^n\times{A^*}^s)/j(B)$ be the natural surjection, so that $B^* = ({A^*}^n\times A^s)/C^*$. The pull-back by $\varpi$ of a differential form $\varphi$ of the first or the second kind on $C$ is a differential form of the same kind on $A^n\times{A^*}^s$ and its integral on a path lying in $\Lie B_\CC$ is zero. Expressing $\varpi^*\varphi$ on the chosen bases of $\mathrm H_{\mathrm{dR}}^1(A)$ and $\mathrm H_{\mathrm{dR}}^1(A^*)$ modulo exact forms, the equation $\int_0^{(p,q)}\varpi^*\varphi=0$ gives a linear relation with coefficients in $K$ between the components of any generalized logarithms $\widetilde{\log}_A(P)=(\int_{z}^{p_i+z}\omega_j,\int_{z}^{p_i+z}\eta_j)_{i,j}$ and $\widetilde{\log}_{A^*}(Q)= (\int_{z}^{q_k+z}\eta_j,-\int_{z}^{q_k+z}\omega_j)_{k,j}$.
\textit{Letting $\varphi$ run over a base of $\mathrm H_{\mathrm{dR}}^1(C)$ we find a basis of $2\big((n+s)g-\dim(B)\big)$ linear equations satisfied by any generalized logarithm $\tlog_B(b)$ of any point $b$ belonging to $B$}.

By Lemma \ref{lemma[]}, $\Lie Z'(1)_\CC$ is the smallest Lie sub-algebra of $\Lie \GG_{m \, \CC}^{ns}  $ which contains all logarithms whose coordinates matrices are of the kind
$u v^* - v u^* $, where
$(u,u^{*})$ and $(v,v^{*})$ are respectively coordinates matrices of logarithms $ \log_B(b_1),\log_B(b_2)$ of two points $b_1,b_2$ of $B$ such that
 $\log_B(b_1) $ and $ \log_B(b_2) $ belong to
$\HH_1(B(\CC),\QQ). $ 
The Lie sub-algebra $\Lie Z(1)_\CC$ is the smallest Lie sub-algebra of $ \Lie \GG_{m\, \CC}^{ns} $ containing $\Lie Z'(1)_\CC$ and such that $\Lie (Z/Z')(1)_\CC$ contains the image of the point $\Xi_R$. Hence\textit{ $\Lie Z(1)_\CC$ is defined by a set of $ns-\dim Z(1)$ linear equations with coefficients in $\mathbb Q$ linking the entries of $u v^* - v u^*$ and the components of  $\Xi_R \; \mathrm{mod} \,\Lie Z'(1)_\CC$.} The space of these linear equations is isomorphic to the kernel of \eqref{pointpsi}.
\end{remark}

\section{Example: motives over semi-elliptic surfaces}\label{subsubsection4}	Let $M=[\ZZ \to G]$ be a 1-motive defined over $\oQQ$, where $G$ is an extension of an elliptic curve $\cE$ by $\GG_m$ parametrized by a point  $Q \in \cE^*(\oQQ)$, and where the group morphism $\ZZ \to G$ is defined by a point $R=(P, e^\ell)  \in G(\oQQ)$ living above a point $P\in \cE(\oQQ)$. Set $r=(p, \log f_q(p)- \ell) $ be a logarithm of the point $R$. With notations of Examples \ref{example:matriceMsemi-elliptique} and \ref{example:GalMotE}, assuming that $\cE$ has complex multiplication, we have the action
\begin{multline*} \rho(a,u,u^*,\sigma)\Pi_M = \left(\begin{matrix}
		1 &u_1 &u_2 &\sigma\\
		0 &\frac{a_2\tau-a_1\bar\tau}{\tau-\bar\tau} &\frac{a_1-a_2}{\tau-\bar\tau} &u_1^*\\
		0 &\frac{(a_2-a_1)|\tau|^2}{\tau-\bar\tau} &\frac{a_1\tau-a_2\bar\tau}{\tau-\bar\tau} &u_2^*\\
		0 &0 &0 & a_1 a_2
	\end{matrix}\right)
	\left(\begin{matrix}
		\log f_q(p)- \ell &p &\zeta(p) &1\\
		\eta_1q-\omega_1\zeta(q) &\omega_1 &\eta_1 &0\\
		\eta_2q-\omega_2\zeta(q) &\omega_2 &\eta_2 &0\\
		2i\pi &0 &0 &0
	\end{matrix}\right) = \\
	 = \left(\begin{matrix}
		\log f_q(p) - \ell +q\eta(u)-\zeta(q)\omega(u)+2i\pi\sigma &p+\omega(u) &\zeta(p)+\eta(u) &1\\
		q(a_2\eta_1-\frac{(a_1-a_2)\kappa\omega_1}{(\tau-\bar\tau)\tau})-\zeta(q)a_1\omega_1+2i\pi u_1^* &a_1\omega_1 &a_2\eta_1-\frac{(a_1-a_2)\kappa\omega_1}{(\tau-\bar\tau)\tau} &0\\
		q(a_2\eta_2-\frac{(a_1-a_2)\kappa\omega_2}{(\tau-\bar\tau)\tau})-\zeta(q)a_1\omega_2+2i\pi u_2^* &a_1\omega_2 &a_2\eta_2-\frac{(a_1-a_2)\kappa\omega_2}{(\tau-\bar\tau)\tau} &0\\
		2i\pi a_1 a_2 &0 &0 &0
	\end{matrix}\right)
\end{multline*}
	with $u=(u_1,u_2) , u^*=(u_1^*,u_2^*) \in \GG_a(\QQ)^2,$ $\sigma \in \GG_a(\QQ),$ and where we set $\omega(u):=u_1\omega_1+u_2\omega_2$ and $\eta(u):=u_1\eta_1+u_2\eta_2 .$ By the action of $ \rho(a,u,u^*,\sigma)$ on the periods of $M$ we understand the transformation of a coefficient of $\Pi_M$ in the corresponding coefficient of the product matrix 
	$ \rho(a,u,u^*,\sigma)\Pi_M.$ By pull-back we deduce the action of $ \rho(a,u,u^*,\sigma)$ on polynomial relations between periods of $M$. The coordinates matrix $u \Upsilon_Q$ appearing in the matrix \eqref{actionGalmot} corresponds to the term $q\eta(u)- \zeta (q)\omega(u)$ of the above matrix: in fact by Lemma \ref{computperiodes}
$ u \Upsilon_Q = - u \Pi_A (\tlog_{A^*}(Q))^t = - \big(\omega(u), \eta(u)\big) \big(\zeta(q), -q \big)^t.$

 \medskip

In the case of our elliptic curve $\cE$, the principal polarization is $H(z_1,z_2) = \frac{z_1 {\bar z}_2}{ \mathrm{Im}(\omega_1 \overline{\omega}_2)}$ and the matrix of the alternating form $\mathrm{Im} H$ is 
$J= \left( \begin{matrix}
	0&1 \\
	 -1& 0 
\end{matrix}\right) $. Define $\omega(u^*):=u_1^*\omega_1+u_2^*\omega_2$ and $\eta(u^*):=u_1^*\eta_1+u_2^*\eta_2 .$  The point $\big(\omega(u), \phi_H (\omega(u^*))\big)$ runs over the smallest Lie subalgebra $\Lie B_\CC$ of $\Lie \cE_\CC \times \Lie \cE^*_\CC$ which contains a multiple of $(p, \phi_H(q))$ modulo $\Lambda \times \Lambda^*$ if and only if $(u,u^*) $ runs in $\HH_1 (B(\CC), \QQ)$. By the last two equalities of \eqref{eq:ratio}, the value of the Lie realization of the motivic Weil pairing at a point of $\Lie B_\CC$ is 
\[\mathcal{W}_A^{\Lie}\big(\omega(u),\phi_H(\omega(u^*))\big) = e^{ 2{\rm i} \pi \mathrm{Im}\big(H(\omega(u^*),\omega(u))\big)} = e^{\frac{2{\rm i} \pi}{\mathrm{Im}(\omega_1 \overline{\omega}_2)} \mathrm{Im} \big(\omega(u^*) \overline{\omega(u)} \big)}.\]
By \eqref{eq:LieRealization[,]}, the value of the Lie realization of the Lie Bracket $[\cdot, \cdot]: B \otimes B \to \GG_m$ at a point  of $\Lie B_\CC^2$ is 
\begin{equation}\label{[]casE} \big[\big(\omega(u), \phi_H (\omega(u^*))\big),\big(\omega(v), \phi_H (\omega(v^*))\big) \big] = e^{ \frac{ 2{\rm i} \pi}{\mathrm{Im}(\omega_1 \overline{\omega}_2)} \mathrm{Im} \big(\omega(v^*) \overline{\omega(u)} + \omega(v) \overline{\omega(u^*)} \big) }.
\end{equation}
Hence the entry $\sigma$ of the matrix $\rho(a,u,u^*,\sigma)$ runs over the smallest $\QQ$-subspace of $\Lie \GG_m$ which contains all the logarithms of the above values
of the Lie realization of the Lie Bracket $[\cdot, \cdot]: B \otimes B \to \GG_m$ and a multiple of the period $\log f_q(p) - \ell $ modulo $2 {\rm i} \pi \ZZ$.

 \medskip

According to the geometry of the 1-motive $M$ (see results of Section \ref{studyDimUR}), we now list
\begin{itemize}
	\item polynomial relations with coefficients in $\oQQ$ between the periods of $M$, which come from algebraic cycles on $M$ and on the products of $M$ with itself. These polynomial relations generate an ideal $\mathcal I$ whose locus of zeros $Z(\mathcal I)$ contains the point $\Pi_M$;
	\item a set of generators of the field $\oQQ(\Pi_M)$;
	\item a parametrization of the subgroup $G_{\mathcal I}$ of $\mathrm{Gl}_4(\QQ)$ which leaves invariant the ideal $\mathcal I$. 
\end{itemize} 

 \textbf{From now on and until the end of this text, we assume Grothendieck periods conjecture for the 
  1-motive $M$ defined over $\oQQ$.} It implies that the ideal $\mathcal I$ contains all polynomial relations between periods of $M$ and in particular it is prime. The subgroup $G_{\mathcal I}$ of $\mathrm{Gl}_4(\QQ)$ contains then $\MT(M)$. Finally we check that $\dim G_{\mathcal I} = \dim \Galmot (M) $, and therefore we can conclude that $G_{\mathcal I}$ is a representation of the Mumford-Tate group of $M$. Under our assumption, the transcendence degree of the field $\oQQ(\Pi_M)$ over $\oQQ$ is equal to $10- \mathrm{Rank}(\mathcal I) ,$ that is
 \[\mathrm{Numbers \; of \; periods \; of \;} M - \mathrm{Rank}(\mathcal I) =\dim \Galmot(M). \]
 In other words a decrease in the dimension of the motivic Galois group of $M$ is equivalent to an increase of the rank of the ideal $\mathcal I$. 
  \medskip
  
  We start with polynomial relations between the periods of the elliptic curve $\cE$.
The maximal dimension of the motivic Galois group of an elliptic curve is 4 and it occurs if the elliptic curve has no complex multiplication: in fact, in this case $\MT(\cE) =\mathrm{Gl}_2(\QQ)$ and we have no polynomial relations between the 1-periods of $\cE$. If $\cE$ has complex multiplication, the dimension of $\Galmot (\cE)$ is 2: in fact, in this case $\MT(\cE) \cong \GG_m^2$ and it exists $2=4-2$ polynomial relations between the periods of $\cE$: 
\begin{equation}\label{PolynomialRelationsForE}
\left[\begin{array}{l}
 \omega_2-\tau\omega_1 =0,\\
\bar\tau\eta_1-\eta_2-\frac{\kappa}{\tau}\omega_1 =0 \\
\end{array}\right.
\end{equation}
with $\kappa$ belonging to the field of definition of $\cE$. Using these two polynomials, in Example \ref{example:GalMotE} we compute a $2 \times 2$ invertible matrix which represents an element of $\MT(\cE).$

Observe that Legendre relation is not a polynomial relation between the 1-periods of $\cE$, since $2 \mathrm i \pi$ is a 2-period of $\cE.$ But as observed in Example \ref{legendreRelation}, Legendre relation expresses the 2-period  $2 \mathrm i \pi$ as the value of a degree 2 polynomial evaluated in 1-periods (see Lemma \ref{lemma:r-periods}).

 \medskip

In order to recover Legendre relation, consider the 1-motive $M= \cE \times \GG_m.$
 We have $\dim \Galmot(M)=  \dim \Galmot(\cE)$ and, in terms of matrices, $\Galmot(M)$ embeds in $\Galmot (\cE) \times \Galmot(\GG_m) \subseteq \mathrm{Gl}_2(\QQ) \times \GG_m$ via $a \mapsto (a, \det a).$ Observe that $\dim (\Galmot (\cE) \times \Galmot(\GG_m)) - \dim \Galmot (M) =1$ and this decrease in the dimension
 of $\Galmot(M)$ corresponds to Legendre relation. Hence if $\cE$ has no complex multiplication, we have just Legendre relation as polynomial relation between the periods of $M$, and if $\cE$ has complex multiplication, we have $3$ polynomial relations between the periods of $M= \cE \times \GG_m$: 
 \begin{equation}\label{PolynomialRelationsForEG_m}
 \left[\begin{array}{l}
 	\omega_2-\tau\omega_1 =0, \\
 	\bar\tau\eta_1-\eta_2-\frac{\kappa}{\tau}\omega_1 =0, \\
 	 2i\pi -\omega_1\eta_2+\omega_2\eta_1=0.\\
 \end{array}\right.
 \end{equation}

\medskip

Now let $M=[\ZZ \to G]$ be as at the beginning of this Section. By \eqref{eq:shortexactsequenceUR} and \cite[Lemma 3.5]{BPSS},
$\dim \Galmot(M)=  \dim \Galmot(\cE)  + \dim \UR(M),$
 with $\dim \UR(M) = 2 \dim B + \dim Z(1).$   Look at \cite[Table 2]{BPSS} and Corollary \ref{cor:DimSemi-ellipticCase} for a complete table of the dimensions of the motivic Galois group of $M$. 
By Corollary \ref{cor:maxdim} the dimension of the unipotent radical of $\Galmot(M)$ is maximal if the points $P$ and $Q$ are $\End(\cE) \otimes_\ZZ \QQ$-linearly independent: under these hypothesis, $
\dim \UR(M)=5$. If the dimension of $\UR(M)$ is not maximal, we have 
$ 5-\dim \UR(M) $ polynomial relations coming from geometry between the periods of $M$ (recall we are under Grothendieck periods conjecture), that add to Legendre relation if $\cE$ has no complex multiplication, or that add to the three polynomials \eqref{PolynomialRelationsForEG_m} if $\cE$ has complex multiplication. We write down these $ 5-\dim \UR(M)$ polynomial relations according to the geometry of $M$ (see results of Section \ref{studyDimUR}).

\medskip

  Recall that we have the short exact sequence $ 0 \to \UR (\MT(M)) \to \MT(M) \to \MT(\cE) \to 0$, which is the 
   image of the exact sequence \eqref{eq:shortexactsequenceUR} via the fibre functor ``Hodge realization" $ \T_\HH$. In each case we find a section $F:\MT(\cE)\rightarrow \MT(M)$ that identifies $\MT(\cE)$ with an (abelian if $\cE$ has complex multiplication) subgroup of matrices in $\MT(M)$, which leave invariant the ideal of relations $\mathcal I$ and has trivial intersection with the normal subgroup $\UR(\MT(M))$. By Mostow's Theorem (see \cite[Theorem page 200]{Mostow56}) the Mumford-Tate group $\MT(M)$ is isomorphic to the (inner) semi-direct product $F(\MT(\cE)) \ltimes \UR(\MT(M))$:
 \begin{equation}\label{semi-direct}
 	\begin{matrix}
 		F(\MT(\cE)) \ltimes \UR(\MT(M)) &\longrightarrow & \MT(M)\hfill\\
 		\hfill \big(F(a) , \rho(\mathrm{Id}_2,u,u^*,\sigma)\big) &\longmapsto &F(a)\rho(\mathrm{Id}_2,u,u^*,\sigma).
 	\end{matrix}
 \end{equation}
 Observe that the product $	(F(a),\rho(\mathrm{Id}_2,u,u^*,\sigma))\cdot (F(b),\rho(\mathrm{Id}_2,v,v^*,\delta))$, which is defined as 
  \[ \big( F(ab)\, ,\, F(b)^{-1}\rho(\mathrm{Id}_2,u,u^*,\sigma)F(b) \rho(\mathrm{Id}_2,v,v^*,\delta) \big),
 \]
 corresponds in each case to the group law \eqref{eq:grouplaw} of $\MT(M)$, and so the arrow \eqref{semi-direct} is an homomorphism.

  In order to write explicitly in each cases below this decomposition \eqref{semi-direct} of $\MT(M),$ we compute the subgroup of $\mathrm{Gl}_4(\oQQ)$ which leaves invariant the ideal of algebraic relations between periods and its unipotent radical. By Mostow's Theorem we know that there exists a (unique up to conjugation by an element of the unipotent radical) reductive subgroup $F(\MT(\cE))$ such that  \eqref{semi-direct} holds. Hence we have just to check that the group $F(\MT(\cE))$  given in each cases below satisfies the following conditions:
\begin{enumerate}
	\item $F(\MT(\cE)) \cap \UR(\MT(M)) =\{ \mathrm{Id}\} ,$
	\item $F(\MT(\cE))$ leaves invariant the ideal of relations between periods,
	\item the semi-direct product $	F(\MT(\cE)) \ltimes \UR(\MT(M))$ in (6.4) is the whole $\MT(M).$
	   \end{enumerate}
 \medskip

Set $\alpha=(\alpha_1,\alpha_2):=(p,\zeta(p))\Pi_\cE^{-1}$, $\beta=(\beta_1, \beta_2) := (q,\zeta(q))\Pi_\cE^{-1}$ and $\gamma:= (\log f_q(p) - \ell) /2\mathrm{i}\pi$. These vectors have rational coordinates if and only if respectively $P$, $Q$ and $\pi(\tilde R)$ is torsion.

\counterwithout{subsection}{section}

\setcounter{subsection}{1}
\subsection{\boxed{\dim B=2} \emph{that is $P$ and $Q$ are $\End(\cE) \otimes_\ZZ \QQ$-linearly independent}}\phantom{.}

\par\noindent By Corollary \ref{cor:DimSemi-ellipticCase} \emph{(4)} we have that $\dim B=2$ and $\dim Z(1)= \dim Z'(1)=1$, that is
 $\dim \UR(M)=5$. In particular the unipotent radical of $\Galmot(M)$ does not furnish extra relations between the periods of $M$.
Hence as polynomial relations we have just Legendre relation if $\cE$ has no complex multiplication, or the 3 polynomials stated in \eqref{PolynomialRelationsForEG_m} if $\cE$ has complex multiplication. Observe that the equality $\dim (Z/Z')(1)=0$ implies that $\pi(pr_*\tilde R)$ is necessarily torsion, but this does not bring any new algebraic relation between the periods of $M$ !

\par\noindent A transcendence base of the field $\oQQ(\Pi_M)$ is given by the nine periods
\[\omega_1,\;  \eta_1, \; \omega_2,\; \eta_2,\;  p,\;  \zeta(p), \;q, \; \zeta(q), \;\log f_q(p)-\ell\]
if $\cE $ has no complex multiplication, by the seven periods 
\[\omega_1, \;  \eta_1, \; p, \;  \zeta(p), \; q, \; \zeta(q), \;\log f_q(p)-\ell\]
if $\cE$ has complex multiplication. 
 
 \par\noindent The Mumford-Tate group $\MT(M)$ of $M$ is represented by the matrices as in \eqref{semi-direct} with
\smallskip

 $\left\{\begin{array}{l}
   a \; \mathrm{any}\;  2 \times 2 \; \mathrm{matrix \; in} \; \MT(\cE)(\oQQ) \mathrm{(see \; Example \;\ref{example:GalMotE})}  \;
   \mathrm{and} \;  F(a) \;  \mathrm{as \; below},\\
  u\in\mathrm{Mat}_{1,2}(\oQQ), u^*\in\mathrm{Mat}_{2,1}(\oQQ), \sigma \in \oQQ \; \mathrm{five \; free \; parameters}.\\
 \end{array}\right.$

\par\noindent In \eqref{semi-direct} an element of $\MT(M)$ has the following representation
\begin{equation*}
	\begin{matrix}
	\hfill	F(\MT(\cE)) \ltimes \UR(\MT(M)) &\longrightarrow & \MT(M)\hfill\\[2mm]
		\hfill \begin{pmatrix}
			1 &0 &0\\
			0 &a &0\\
			0 &0 &\det(a)
		\end{pmatrix}
		\begin{pmatrix}
			1 &u & \sigma \\
			0 &\mathrm{Id}_2 &u^*\\
			0 &0 &1
		\end{pmatrix} &\longmapsto &\rho(a,u,au^*,\sigma).
	\end{matrix}
\end{equation*}

\smallskip

The cases of elliptic curves with or without complex multiplication are very similar (see \cite[Table 2]{BPSS}). However, if the elliptic curve $\cE$ has complex multiplication the geometry of the 1-motive is richer (see the cases \ref{anti}). Hence,
{\bf from now on we restrict to the case with complex multiplication}. Recall that by default we then have the three polynomial relations \eqref{PolynomialRelationsForEG_m} given by the reductive part of $\Galmot(M)$.	
\medskip  

\setcounter{subsection}{0}
\subsection{\boxed{\dim B=1}  \emph{that is $P$ and $Q$ are $\End(\cE) \otimes_\ZZ \QQ$-linearly dependent, but not both torsion}}\phantom{.}

\par\noindent We distinguish seven subcases. 

\subsubsection{\bf $P$ and $Q$ are $\End(\cE) \otimes_\ZZ \QQ$-linearly dependent, but none is torsion}\phantom{.}

\par\noindent It exists an isogeny $\varphi = \varphi_1 +  \varphi_2 \tau  \in \End(\cE) \otimes_\ZZ \QQ =  \QQ \oplus \tau \QQ$ such that $\varphi(p)=q$ with $\phi_H(q)$ a logarithm of $Q$. By \cite[formula (37) at page 37]{M75} this action is represented in coordinates $(z, \zeta(z))$ through multiplication on the right by the matrix $\Phi = \left(\begin{matrix} \varphi& - \frac{\kappa}{\tau} \varphi_2  \\[1mm]  0 & \overline\varphi \end{matrix} \right)$, so that $\beta = \alpha\Pi_{\cE}\Phi\Pi_\cE^{-1}$. If $\big(\omega(u), \phi_H (\omega(u^*))\big)$ and $\big(\omega(v), \phi_H (\omega(v^*))\big)$ are two points of $\Lie B_\CC$, we have that  $\omega(u^*) = \varphi (\omega(u))$ and 
$\omega(v^*) = \varphi (\omega(v))$ and hence by \eqref{[]casE} we get 
\[ \big[\big(\omega(u), \phi_H (\omega(u^*))\big),\big(\omega(v), \phi_H (\omega(v^*))\big) \big] = e^{ 2{\rm i} \pi \frac{ \mathrm{Im} ( \varphi\omega(v) \overline{\omega(u)} + \omega(v) \overline{ \varphi \omega(u)} )}{\mathrm{Im}(\omega_1 \overline{\omega}_2)}  } 
= e^{ 2{\rm i} \pi (\varphi + \overline{\varphi})  \frac{\mathrm{Im} ( \omega(v) \overline{\omega(u)} )}{\mathrm{Im}(\omega_1 \overline{\omega}_2)}   }.\]
As predicted by Corollary \ref{cor:DimSemi-ellipticCase} \emph{(3)}, these equalities show that the dimension of the smallest torus $ Z'(1),$ which contains the image of the  Lie Bracket $[\cdot, \cdot]: B \otimes B \to \GG_m,$ is 0 or 1 according to the fact that the endomorphism $\varphi$ is antisymmetric (that is $ \varphi + \overline{\varphi} =0$ or equivalently $\Phi+\det(\Phi)\Phi^{-1}=0$) or not.

 We have three subcases:
\paragraph{\underline{$q= \varphi(p)$ with $\varphi$ a non antisymmetric endomorphism}}\phantom{.}

\par\noindent By Corollary \ref{cor:DimSemi-ellipticCase} \emph{(3)(b)} we know that $\dim B=1$ and $ \dim Z(1)=\dim Z'(1)=1$, which implies that $\dim \UR(M)= 3$ and the point
 $ \pi (pr_* \tR)$ defined in \eqref{piR'} is torsion.
 As extra relations between the periods of $M$, the abelian quotient $B$ of the unipotent radical of $\Galmot(M)$ furnishes the $2$ polynomials	
$$ \left[\begin{array}{l}
- \omega_2 (\eta_1q-\omega_1\zeta(q)) + \omega_1(	\eta_2q-\omega_2\zeta(q)) - 2 \mathrm i \pi p \varphi= 0, \\
  - \eta_2 (\eta_1q-\omega_1\zeta(q)) + \eta_1(\eta_2q-\omega_2\zeta(q)) - 2 \mathrm i \pi \left(\zeta(p) \overline \varphi- p\frac{\kappa}{\tau}\varphi_2 \right)  \in \oQQ.
\end{array}\right.$$
Observe that these two equalities express the identity $(q, \zeta(q)) = ( p, \zeta(p))    \Phi $ as relations between periods.

\par\noindent The rank of $\mathcal I$ is 5 and a base of the field $\oQQ(\Pi_M)$ is given by the five periods
\[\omega_1, \; \eta_1, \;  p, \;  \zeta(p), \; \log f_q(p)-\ell.\]

\par\noindent The Mumford-Tate group $\MT(M)$ of $M$ is represented by the matrices as in \eqref{semi-direct} with the relation $u= (Ja^{-1}u^*)^t \Pi_\cE \Phi^{-1} \Pi_\cE^{-1}.$ Setting $\lambda= u \Pi_\cE \Phi = (a^{-1}u^*)^t J^t \Pi_\cE,$ we have
\smallskip

$\left\{\begin{array}{l}
	a \; \mathrm{any}\;  2 \times 2 \; \mathrm{matrix \; in} \; \MT(\cE) (\oQQ)  \;
	\mathrm{and} \;  F(a) \;  \mathrm{as \; below},\\
 u= 	\lambda   (\Pi_\cE \Phi)^{-1},\\
 u^*=  (\lambda \Pi_\cE^{-1} J)^{t},\\
	 \lambda \in \mathrm{Mat}_{1,2}(\oQQ) , \sigma  \in \oQQ \; \mathrm{three} \; \mathrm{free \; parameters}.\\
\end{array}\right.$

\par\noindent In \eqref{semi-direct} an element of $\MT(M)$ has the following representation
\begin{equation*}
	\begin{matrix}
		\hfill	F(\MT(\cE)) \ltimes \UR(\MT(M)) &\longrightarrow & \MT(M)\hfill\\[2mm]
		\hfill \begin{pmatrix}
			1 &0 &0\\
			0 &a &0\\
			0 &0 &\det(a)
		\end{pmatrix}
		\begin{pmatrix}
			1 &	\lambda  (\Pi_\cE \Phi)^{-1} & \sigma \\
			0 &\mathrm{Id}_2 &  (\lambda \Pi_\cE^{-1} J)^{t}  \\
			0 &0 &1
		\end{pmatrix} &\longmapsto &\rho(a,\lambda  (\Pi_\cE \Phi)^{-1}, a (\lambda \Pi_\cE^{-1} J)^{t},\sigma).
	\end{matrix}
\end{equation*}

\medskip

\paragraph{\underline{$q= \varphi(p)$ with $\varphi$ an antisymmetric endomorphism}}\phantom{.}\label{anti}
 
\par\noindent Recall that this case occurs \textbf{only if $\cE$ has complex multiplication.} By Corollary \ref{cor:DimSemi-ellipticCase} \emph{(3)(a)}, $\dim B=1$ and $\dim Z'(1)=0$. According to Lemma \ref{lemma:pr*i*d*B-trivial} the torsor $i^*d^* \mathcal{B}$ is trivial, that is $i^*d^* \mathcal{B}=B \times \GG_m$, and so by \eqref{expG} $\pi (pr_* \tR) =\pi( \tilde R)= e^\ell.$ We have two subcases:
	\medskip

\subparagraph{\underline{\emph{$\pi(\tilde R)$ isn't a torsion point}}}\phantom{.}
 
\par\noindent Then $\dim Z(1) =\dim (Z/Z')(1)=1$ and  $\dim \UR(M)= 3$. As in the previous case, the abelian quotient $B$ of the unipotent radical of $\Galmot(M)$ furnishes the $2$ polynomial relations	
$$ \left[\begin{array}{l}
- \omega_2 (\eta_1q-\omega_1\zeta(q)) + \omega_1(	\eta_2q-\omega_2\zeta(q)) - 2 \mathrm i \pi p \varphi= 0, \\
- \eta_2 (\eta_1q-\omega_1\zeta(q)) + \eta_1(\eta_2q-\omega_2\zeta(q)) - 2 \mathrm i \pi \left(\zeta(p) \overline \varphi- p\frac{\kappa}{\tau}\varphi_2 \right)  \in \oQQ.
\end{array}\right.$$
As observed in the previous case, these two equalities express the identity $(q, \zeta(q)) = ( p, \zeta(p))    \Phi $ as relations between periods.

\par\noindent The rank of $\mathcal I$ is 5 and a base of the field $\oQQ(\Pi_M)$ is given by the five periods
\[\omega_1, \; \eta_1, \;  p, \;  \zeta(p), \; \log f_q(p)-\ell.\]

\par\noindent The Mumford-Tate group $\MT(M)$ of $M$ is represented by the matrices as in \eqref{semi-direct} with the relation  $u= (Ja^{-1}u^*)^t \Pi_\cE \Phi^{-1} \Pi_\cE^{-1}.$ Setting $\lambda= u \Pi_\cE \Phi = (a^{-1}u^*)^t J^t \Pi_\cE,$ we have 
\smallskip

$\left\{\begin{array}{l}
	a \; \mathrm{any}\;  2 \times 2 \; \mathrm{matrix \; in} \; \MT(\cE)(\oQQ)  \;
	\mathrm{and} \;  F(a) \;  \mathrm{as \; below},\\
	u= 	\lambda  (\Pi_\cE \Phi)^{-1},\\
	u^*=  (\lambda \Pi_\cE^{-1} J)^{t},\\
 \lambda \in \mathrm{Mat}_{1,2}(\oQQ) , \sigma  \in \oQQ \; \mathrm{three} \; \mathrm{free \; parameters}.\\
\end{array}\right.$

\par\noindent In \eqref{semi-direct} an element of $\MT(M)$ has the following representation
\begin{equation*}
	\begin{matrix}
		\hfill	F(\MT(\cE)) \ltimes \UR(\MT(M)) &\longrightarrow & \MT(M)\hfill\\[2mm]
		\hfill \begin{pmatrix}
			1 &0 &0\\
			0 &a &0\\
			0 &0 &\det(a)
		\end{pmatrix}
		\begin{pmatrix}
			1 &	\lambda  (\Pi_\cE \Phi)^{-1} & \sigma \\
			0 &\mathrm{Id}_2 &  (\lambda \Pi_\cE^{-1} J)^{t}\\
			0 &0 &1
		\end{pmatrix} &\longmapsto &\rho(a,\lambda  (\Pi_\cE \Phi)^{-1}, a (\lambda \Pi_\cE^{-1} J)^{t},\sigma).
	\end{matrix}
\end{equation*}

 \medskip
	
\subparagraph{\underline{\emph{$\pi( \tilde R)$ is a torsion point}}}\phantom{.}

\par\noindent Then $\dim Z(1) =\dim (Z/Z')(1)=0$ and $\dim \UR(M)= 2$. Both the abelian quotient $B$ and the toric part $Z(1)$ of the unipotent radical of $\Galmot(M)$ furnishes the $3$ polynomials	
$$ \left[\begin{array}{l}
- \omega_2 (\eta_1q-\omega_1\zeta(q)) + \omega_1(	\eta_2q-\omega_2\zeta(q)) - 2 \mathrm i \pi p \varphi= 0, \\
- \eta_2 (\eta_1q-\omega_1\zeta(q)) + \eta_1(\eta_2q-\omega_2\zeta(q)) - 2 \mathrm i \pi \left(\zeta(p) \overline \varphi- p\frac{\kappa}{\tau}\varphi_2 \right)  \in \oQQ, \\
	\log f_q(p) - \ell -2i\pi \gamma =0.
\end{array}\right.$$
As before, the first two equalities express the identity $(q, \zeta(q)) = ( p, \zeta(p))    \Phi $
and the last equation expresses that $\pi( \tilde R)$ is a torsion point in $\GG_m$, as relations between periods.
By \eqref{eq:grouplaw}, applying the action $\rho(\mathrm{Id}_2,v,v^*,\tau) \rho(a,u,u^*,\sigma)$ on the last polynomial, we observe that $uv^*$ must be zero. This property appears each time we have this last polynomial independently of its geometrical origin. 

\par\noindent  The rank of $\mathcal I$ is 6 and a base of the field $\oQQ(\Pi_M)$ is given by the four periods
\[\omega_1, \; \eta_1,\;  p, \; \zeta(p).\]

\par\noindent The Mumford-Tate group $\MT(M)$ of $M$ is represented by the matrices as in \eqref{semi-direct} with the relation  $u= (Ja^{-1}u^*)^t \Pi_\cE \Phi^{-1} \Pi_\cE^{-1}.$ Setting $\lambda= u \Pi_\cE \Phi = (a^{-1}u^*)^t J^t \Pi_\cE,$ we have 

$\left\{\begin{array}{l}
	a \; \mathrm{any}\;  2 \times 2 \; \mathrm{matrix \; in} \; \MT(\cE)(\oQQ) \;
	\mathrm{and} \;  F(a) \;  \mathrm{as \; below},\\
		u= 	\lambda  (\Pi_\cE \Phi)^{-1},\\
	u^*=  (\lambda \Pi_\cE^{-1} J)^{t},\\
	\sigma =  \frac{1}{4 \mathrm{i} \pi}  \lambda \Phi^{-1} J^t \lambda^t ,\\
	\lambda \in \mathrm{Mat}_{1,2}(\oQQ)  \; \mathrm{two} \; \mathrm{free \; parameters}.\\
\end{array}\right.$

\par\noindent An element of the semi-direct product $F(\MT(\cE)) \ltimes \UR(\MT(M))$ has the following representation 
\begin{equation*}
 \begin{pmatrix}
			1 &0 &(\det(a)-1)\gamma\\
			0 &a &0\\
			0 &0 &\det(a)
		\end{pmatrix}
		\begin{pmatrix}
			1 &	\lambda  (\Pi_\cE \Phi)^{-1} &  \frac{1}{4 \mathrm{i} \pi} \lambda \Phi^{-1} J^t \lambda^t\\
			0 &\mathrm{Id}_2 &   (\lambda \Pi_\cE^{-1} J)^{t}\\
			0 &0 &1
		\end{pmatrix}
\end{equation*}
and its image in $\MT(M)$ via \eqref{semi-direct} is the element 
\[\rho\big(a,	\lambda  (\Pi_\cE \Phi)^{-1}, a   (\lambda \Pi_\cE^{-1} J)^{t} ,  \frac{1}{4 \mathrm{i} \pi}\lambda \Phi^{-1} J^t \lambda^t + (\det(a)-1)\gamma \big).\]

\medskip

\subsubsection{\bf{$P$ is a torsion point, but not $Q$}}\phantom{.}

\par\noindent According to Corollary \ref{cor:DimSemi-ellipticCase} \emph{(2)}, $\dim B=1$ and $\dim Z'(1)= 0.$ By Lemma \ref{lemma:pr*i*d*B-trivial} the torsor $i^*d^* \mathcal{B}$ is trivial, that is $i^*d^* \mathcal{B}=B \times \GG_m,$ and so $\pi (pr_* \tR) =\pi( \tilde R)= e^\ell.$ We have two subcases:
	\medskip
	
\paragraph{\underline{\emph{$\pi( \tilde R)$ is not a torsion point}}}\phantom{.} 

\par\noindent Then $\dim Z(1)=\dim (Z/Z')(1)=1$, that is $ \dim \UR(M)= 3$. The abelian quotient $B$ of the unipotent radical of $\Galmot(M)$ furnishes the $2$ polynomial relations 	
$$ \left[\begin{array}{l}
 p-\alpha_1\omega_1-\alpha_2\omega_2=0,  \\
\zeta(p)-\alpha_1\eta_1-\alpha_2\eta_2   \in \oQQ. 
\end{array}\right.$$
These equations express that $P$ is a torsion point of the elliptic curve $\cE,$ as relations between periods.

\par\noindent  The rank of $\mathcal I$ is 5 and a base of the field $\oQQ(\Pi_M)$ is given by the five periods
\[\omega_1,\;  \eta_1, \;  q, \;  \zeta(q), \; \log f_q(p)-\ell.\]

\par\noindent	The Mumford-Tate group $\MT(M)$ of $M$ is represented by the matrices as in \eqref{semi-direct} with
\smallskip

$\left\{\begin{array}{l}
	a \; \mathrm{any}\;  2 \times 2 \; \mathrm{matrix \; in} \; \MT(\cE)(\oQQ) \;
	\mathrm{and} \;  F(a) \;  \mathrm{as \; below},\\
	u=0, \\
u^* \in  \mathrm{Mat}_{2,1}(\oQQ) , \sigma \in \oQQ \; \mathrm{three \; free \; parameters}.
\end{array}\right.$

 \par\noindent  In \eqref{semi-direct} an element of $\MT(M)$ has the following representation
 \begin{equation*}
 	\begin{matrix}
 		\hfill	F(\MT(\cE)) \ltimes \UR(\MT(M)) &\longrightarrow & \MT(M)\hfill\\[2mm]
 		\hfill \begin{pmatrix}
 			1 &\alpha(a-\mathrm{Id}_2) &0\\
 			0 &a &0\\
 			0 &0 &\det(a)
 		\end{pmatrix}
 		\begin{pmatrix}
 			1 &0 & \sigma \\
 			0 &\mathrm{Id}_2 &u^*\\
 			0 &0 &1
 		\end{pmatrix} &\longmapsto &\rho(a,\alpha(a-\mathrm{Id}_2), a u^*,\sigma + \alpha(a-\mathrm{Id}_2) u^*).
 	\end{matrix}
 \end{equation*}

 \medskip
 
 \paragraph{\underline{\emph{$\pi(\tilde R)$ is a torsion point}}}\phantom{.}
  
\par\noindent Then $\dim Z(1)=\dim (Z/Z')(1)=0$, that is $\dim \UR(M)= 2$. Both the abelian quotient $B$ and the toric part $Z(1)$ of the unipotent radical of $\Galmot(M)$ furnishes the $3$ polynomials 
  $$ \left[\begin{array}{l}
    	 p-\alpha_1\omega_1-\alpha_2\omega_2=0,  \\
    	 \zeta(p)-\alpha_1\eta_1-\alpha_2\eta_2   \in \oQQ, \\
    	 	\log f_q(p) - \ell -2i\pi \gamma =0.
 \end{array}\right.$$
The first two equations express that $P$ is torsion in $\cE$ and the last one that $\pi(\tilde R)$ is torsion in $\GG_m$, as relations between periods.

\par\noindent  The rank of $\mathcal I$ is 6 and a base of the field $\oQQ(\Pi_M)$ is given by the four periods
\[\omega_1,\;  \eta_1, \;  q, \;  \zeta(q).\]

\par\noindent In this subcase we work modulo isogenies and so we may assume $\alpha =0.$ The Mumford-Tate group $\MT(M)$ 
 is represented by the matrices as in \eqref{semi-direct} with
\smallskip

$\left\{\begin{array}{l}
	a \; \mathrm{any}\;  2 \times 2 \; \mathrm{matrix \; in} \; \MT(\cE) (\oQQ) \;
	\mathrm{and} \;  F(a) \;  \mathrm{as \; below},\\
	u=0 , \\
	\sigma = 0 ,\\
	u^* \in  \mathrm{Mat}_{2,1}(\oQQ) \; \mathrm{two \; free \; parameters}.
\end{array}\right.$

\par\noindent An element of the semi-direct product $F(\MT(\cE)) \ltimes \UR(\MT(M))$ has the following representation 
\begin{equation*}
\begin{pmatrix}
	1 &\alpha(a-\mathrm{Id}_2) & (\det(a)-1) \gamma\\
	0 &a &0\\
	0 &0 &\det(a)
\end{pmatrix}
\begin{pmatrix}
	1 &0 & 0 \\
	0 &\mathrm{Id}_2 &u^*\\
	0 &0 &1
\end{pmatrix}
\end{equation*}
and its image in $\MT(M)$ via \eqref{semi-direct} is the element 
\[\rho(a,\alpha(a-\mathrm{Id}_2), a u^*, \alpha(a-\mathrm{Id}_2) u^* +(\det(a)-1) \gamma).\]
 
\medskip

\subsubsection{\bf{$Q$ is a torsion point, but not $P$}}\phantom{.} 

\par\noindent It is the dual case of the previous one. According to Corollary \ref{cor:DimSemi-ellipticCase} \emph{(2)}, $\dim B=1$ and $\dim Z'(1)= 0$. By Lemma \ref{lemma:pr*i*d*B-trivial} the torsor $i^*d^* \mathcal{B}$ is trivial, that is $i^*d^* \mathcal{B}=B \times \GG_m,$ and so  $\pi (pr_* \tR) =\pi( \tilde R)= e^\ell.$ We have two subcases:

\medskip
\paragraph{\underline{\emph{$\pi( \tilde R)$ is not a torsion point}}}\phantom{.}

\par\noindent Then $\dim Z(1)=\dim (Z/Z')(1)=1$, that is $\dim \UR(M)= 3$. The abelian quotient $B$ of the unipotent radical of $\Galmot(M)$ furnishes the $2$ polynomial relations 	
 $$ \left[\begin{array}{l}
- \omega_2 (\eta_1q-\omega_1\zeta(q)) + \omega_1(	\eta_2q-\omega_2\zeta(q))-2i\pi(\beta_1 \omega_1+\beta_2\omega_2) =0,  \\
- \eta_2 (\eta_1q-\omega_1\zeta(q)) + \eta_1(	\eta_2q-\omega_2\zeta(q))  - 2i\pi(\beta_1\eta_1+\beta_2\eta_2) \in  \oQQ. 
\end{array}\right.$$
These equations express that $Q$ is a torsion point of $\cE$, as relations between periods.

\par\noindent  The rank of $\mathcal I$ is 5 and a base of the field $\oQQ(\Pi_M)$ is given by the five periods
\[\omega_1, \; \eta_1, \; p, \; \zeta(p), \; \log f_q(p)-\ell.\]

\par\noindent The Mumford-Tate group $\MT(M)$ of $M$ is represented by the matrices as in \eqref{semi-direct} with
\smallskip

$\left\{\begin{array}{l}
	a \; \mathrm{any}\;  2 \times 2 \; \mathrm{matrix \; in} \; \MT(\cE)(\oQQ) \;
	\mathrm{and} \;  F(a) \;  \mathrm{as \; below},\\
u^*= 0, \\
u \in   \mathrm{Mat}_{1,2}(\oQQ),   \sigma \in \oQQ  \; \mathrm{three \; free \; parameters}.
\end{array}\right.$

	\par\noindent In \eqref{semi-direct} an element of $\MT(M)$ has the following representation
	\begin{equation*}
		\begin{matrix}
			\hfill	F(\MT(\cE)) \ltimes \UR(\MT(M)) &\longrightarrow & \MT(M)\hfill\\[2mm]
			\hfill\begin{pmatrix}
				1 &0 &0\\
				0 &a &(a-\det(a)\mathrm{Id}_2)J \beta^t\\
				0 &0 &\det(a)
			\end{pmatrix}
			\begin{pmatrix}
				1 &u & \sigma \\
				0 &\mathrm{Id}_2 &0\\
				0 &0 &1
			\end{pmatrix} &\longmapsto &\rho(a,u, (a- \det (a) \mathrm{Id}_2 )J\beta^t, \sigma).
		\end{matrix}
	\end{equation*}
 
 		\medskip
		
\paragraph{\underline{\emph{$\pi( \tilde R)$ is a torsion point}}}\phantom{.}

\par\noindent Then $ \dim Z(1)=\dim (Z/Z')(1)=0$, that is $ \dim \UR(M)= 2$. Both the abelian quotient $B$ and the toric part $Z(1)$ of the  unipotent radical of $\Galmot(M)$ furnishes the $3$ polynomials 
$$ \left[\begin{array}{l}
	- \omega_2 (\eta_1q-\omega_1\zeta(q)) + \omega_1(	\eta_2q-\omega_2\zeta(q))-2i\pi(\beta_1 \omega_1+\beta_2\omega_2) =0,  \\
	- \eta_2 (\eta_1q-\omega_1\zeta(q)) + \eta_1(	\eta_2q-\omega_2\zeta(q))  - 2i\pi(\beta_1\eta_1+\beta_2\eta_2) \in  \oQQ, \\
		\log f_q(p) - \ell -2i\pi \gamma =0.
\end{array}\right.$$
The first two equations express that $Q$ is torsion in $\cE$ and the last one that $\pi(\tilde R)$ is torsion in $\GG_m$, as relations between periods.

\par\noindent  The rank of $\mathcal I$ is 6 and a base of the field $\oQQ(\Pi_M)$ is given by the four periods
\[\omega_1,\;  \eta_1, \;  p, \; \zeta(p).\]

\par\noindent	The Mumford-Tate group $\MT(M)$ of $M$ is represented by the matrices as in \eqref{semi-direct} with
	\smallskip
	
	$\left\{\begin{array}{l}
		a \; \mathrm{any}\;  2 \times 2 \; \mathrm{matrix \; in} \; \MT(\cE)(\oQQ) \; \mathrm{and} \;  F(a) \;  \mathrm{as \; below}, \\
		u^*= 0, \\
	 \sigma =  \beta_2 u_1 - \beta_1 u_2 =u J\beta^t ,\\
		u \in   \mathrm{Mat}_{1,2}(\oQQ) \; \mathrm{two \; free \; parameters}.
	\end{array}\right.$		

\par\noindent An element of the semi-direct product $F(\MT(\cE)) \ltimes \UR(\MT(M))$ has the following representation 
\begin{equation*}
	\begin{pmatrix}
		1 &0 &(\det(a)-1) \gamma\\
		0 &a &(a-\det(a)\mathrm{Id}_2) J\beta^t\\
		0 &0 &\det(a)
	\end{pmatrix}
	\begin{pmatrix}
		1 &u & u J\beta^t \\
		0 &\mathrm{Id}_2 &0\\
		0 &0 &1
	\end{pmatrix}
\end{equation*}
and its image in $\MT(M)$ via \eqref{semi-direct} is the element 
\[\rho\big(a,u,\big(a- \det (a)\mathrm{Id}_2 \big)J\beta^t, u J\beta^t + (\det(a)-1)\gamma \big)  \big).\]

\medskip 

\setcounter{subsection}{-1}
\subsection{\boxed{\dim B=0}  \textit{that is $Q$ and $P$ are both torsion points}.} \phantom{.}

\par\noindent According to Corollary \ref{cor:DimSemi-ellipticCase} \emph{(1)}, $\dim B= \dim Z'(1)=0$. By Lemma \ref{lemma:pr*i*d*B-trivial} the torsor $i^*d^* \mathcal{B}$ is trivial, that is $i^*d^* \mathcal{B}=B \times \GG_m,$ and so $\pi (pr_* \tR) =\pi( \tilde R)= e^\ell.$ We have two subcases:

	\medskip
\subsubsection{\underline{\emph{$\pi(\tilde R)$ is not a torsion point}}}\phantom{.}

\par\noindent Then $\dim (Z/Z')(1)= 1$, that is $\dim \UR(M)= 1.$ The abelian quotient $B$ of the  unipotent radical of $\Galmot(M)$ furnishes the $4$ polynomials 
$$ \left[\begin{array}{l}
	 p-\alpha_1\omega_1-\alpha_2\omega_2=0,  \\
  \zeta(p)-\alpha_1\eta_1-\alpha_2\eta_2  \in  \oQQ, \\
 - \omega_2 (\eta_1q-\omega_1\zeta(q)) + \omega_1(	\eta_2q-\omega_2\zeta(q))-2i\pi(\beta_1 \omega_1+\beta_2\omega_2) =0,\\
 - \eta_2 (\eta_1q-\omega_1\zeta(q)) + \eta_1(	\eta_2q-\omega_2\zeta(q))  - 2i\pi(\beta_1\eta_1+\beta_2\eta_2)  \in \oQQ.
\end{array}\right.$$
These equations express that $P$ and $Q$ are torsion points of $\cE$, as relations between periods.

\par\noindent  The rank of $\mathcal I$ is 7 and a base of the field $\oQQ(\Pi_M)$ is given by the three periods
\[\omega_1, \; \eta_1, \; \log f_q(p)-\ell.\]

\par\noindent The Mumford-Tate group $\MT(M)$ of $M$ is represented by the matrices as in \eqref{semi-direct} with
\smallskip

	$\left\{\begin{array}{l}
	a \; \mathrm{any}\;  2 \times 2 \; \mathrm{matrix \; in} \; \MT(\cE)(\oQQ) \;  \mathrm{and} \;  F(a) \;  \mathrm{as \; below},\\
	u=0,\\
	u^*= 0, \\
	\sigma  \in \oQQ \; \mathrm{a \; free \; parameter \; belonging \; to}\; \mathbb{G}_a(\QQ).
\end{array}\right.$

 \par\noindent An element of the semi-direct product $F(\MT(\cE)) \ltimes \UR(\MT(M))$ has the following representation 
 \begin{equation*}
 	\begin{pmatrix}
 		1 &\alpha(a-\mathrm{Id}_2) &\alpha(a-\mathrm{Id}_2) J\beta^t\\
 		0 &a &(a-\det(a)\mathrm{Id}_2) J\beta^t\\
 		0 &0 &\det(a)
 	\end{pmatrix}
 	\begin{pmatrix}
 		1 &0 & \sigma \\
 		0 &\mathrm{Id}_2 &0\\
 		0 &0 &1
 	\end{pmatrix}
 \end{equation*}
 and its image in $\MT(M)$ via \eqref{semi-direct} is the element 
 \[\rho\big(a,\alpha(a-\mathrm{Id}_2),(a- \det (a)\mathrm{Id}_2 )J\beta^t, \sigma + \alpha(a-\mathrm{Id}_2) J\beta^t  \big).\]

 \medskip

\subsubsection{\underline{\emph{$\pi( \tilde R)$ is a  torsion point}}}\phantom{.}

\par\noindent Then $\dim (Z/Z')(1)= 0$, that is $\dim \UR(M)= 0.$ Both the abelian quotient $B$ and the toric part $Z(1)$ of the unipotent radical of $\Galmot(M)$ furnishes the $5$ polynomial relations 
	$$ \left[\begin{array}{l}
			 p-\alpha_1\omega_1-\alpha_2\omega_2=0,  \\
		 \zeta(p)-\alpha_1\eta_1-\alpha_2\eta_2  \in \oQQ, \\
		- \omega_2 (\eta_1q-\omega_1\zeta(q)) + \omega_1(	\eta_2q-\omega_2\zeta(q))-2i\pi(\beta_1 \omega_1+\beta_2\omega_2) =0,\\
		- \eta_2 (\eta_1q-\omega_1\zeta(q)) + \eta_1(	\eta_2q-\omega_2\zeta(q))  - 2i\pi(\beta_1\eta_1+\beta_2\eta_2)   \in \oQQ, \\
			\log f_q(p) - \ell -2i\pi \gamma =0. 
	\end{array}\right.$$
The first four equations express that $P$ and $Q$ are torsion in $\cE$ and the last one that $\pi(\tilde R)$ is torsion in $\GG_m$, as relations between periods.

 \par\noindent  The rank of $\mathcal I$ is 8 and a base of the field $\oQQ(\Pi_M)$ is given by the two periods
 \[\omega_1, \;  \eta_1.\]
 
 \par\noindent The Mumford-Tate group $\MT(M)$ of $M$ is represented by the matrices as in \eqref{semi-direct} with
\smallskip

$\left\{\begin{array}{l}
	a \; \mathrm{any}\;  2 \times 2 \; \mathrm{matrix \; in} \; \MT(\cE)(\oQQ) \; \mathrm{and} \;  F(a) \;  \mathrm{as \; below},\\
	u=0,\\
	u^*= 0, \\
	\sigma  =0.
\end{array}\right.$

\par\noindent In this last case the unipotent radical $\UR(\MT(M))$ consists just of the identity matrix. An element of $F(\MT(\cE))$ has the following representation
\begin{equation*}
		\begin{pmatrix}
			1 &\alpha(a-\mathrm{Id}_2) &(\det(a)-1) \gamma +\alpha(a-\mathrm{Id}_2) J\beta^t\\
			0 &a &(a-\det(a)\mathrm{Id}_2) J\beta^t\\
			0 &0 &\det(a)
		\end{pmatrix}
\end{equation*}
and its image in $\MT(M)$ via \eqref{semi-direct} is the element 
\[\rho\big(a,\alpha(a-\mathrm{Id}_2),(a- \det (a)\mathrm{Id}_2)J\beta^t,(\det(a)-1) \gamma +\alpha(a-\mathrm{Id}_2) J\beta^t  \big).\]

\begin{remark}
	 If $n$ and/or $s$ are bigger than 1, one can use
	\cite[Lemma 2.2, see also Remark 2.4]{B19} in order to find all polynomial relations between the periods of $M.$
\end{remark}

\section*{Erratum}

In the earlier papers \cite{BPSS} and \cite{PSS}, the proofs of the main theorems (Theorem 1 and Theorem 1.6 respectively) require a correction. In place of the element $\ell \in \mathcal L_n \setminus \overline{\mathbb Q}$ which is algebraic over $\mathcal E_m$, we will rather get elements $\ell_1, \ldots, \ell_h \in \mathcal L_n$ which are algebraically independent over $\overline{\mathbb Q}$ but algebraically dependent over $\mathcal E_m$. The desired correction and the subsequent changes have been made in the latest versions on arXiv of the respective articles arXiv:2010.15170v2 and arXiv:1811.05167v1.


\bibliographystyle{plain}

\end{document}